\documentclass[11pt]{amsart}
\usepackage{amssymb, amstext, amscd, amsmath}
\usepackage{xypic, color}

\makeatletter
\def\@cite#1#2{{\m@th\upshape\bfseries%
[{#1\if@tempswa{\m@th\upshape\mdseries, #2}\fi}]}} \makeatother

\theoremstyle{plain}
\newtheorem{theorem}{Theorem}[section]
\newtheorem{corollary}[theorem]{Corollary}
\newtheorem{proposition}[theorem]{Proposition}
\newtheorem{lemma}[theorem]{Lemma}

\theoremstyle{definition}
\newtheorem{definition}[theorem]{Definition}
\newtheorem{example}[theorem]{Example}

\newtheorem{remark}[theorem]{Remark}
\newtheorem*{notation}{Remarks on Notation and Conventions}

\def\bbC{\mathbb C}
\def\bbF{\mathbb F}
\def\bbN{\mathbb N}
\def\bbT{\mathbb T}
\def\bbZ{\mathbb Z}

\newcommand{\fA}{\mathfrak A}     

\newcommand{\sA}{\mathcal A}
\newcommand{\sB}{\mathcal B}
\newcommand{\sC}{\mathcal C}
\newcommand{\sF}{\mathcal F}
\newcommand{\sH}{\mathcal H}
\newcommand{\sI}{\mathcal I}
\newcommand{\sJ}{\mathcal J}
\newcommand{\sK}{\mathcal K}
\newcommand{\sL}{\mathcal L}
\newcommand{\sM}{\mathcal M}
\newcommand{\sN}{\mathcal N}
\newcommand{\sO}{\mathcal O}
\newcommand{\sR}{\mathcal R} 
\newcommand{\sS}{\mathcal S}   
\newcommand{\sT}{\mathcal T}
\newcommand{\sX}{\mathcal X}   

\newcommand{\al}{\alpha}
\newcommand{\be}{\beta}
\newcommand{\si}{\sigma}
\newcommand{\de}{\delta}
\newcommand{\om}{\omega}
\newcommand{\ga}{\gamma}

\newcommand{\ca}{\mathrm{C}^*}
\newcommand{\alg}{\operatorname{alg}}
\newcommand{\Aut}{\operatorname{Aut}}
\newcommand{\ad}{\operatorname{ad}}
\newcommand{\spn}{\operatorname{span}}
\newcommand{\id}{{\operatorname{id}}}
\newcommand{\foral}{\text{ for all }}
\newcommand{\End}{{\operatorname{End}}}
\newcommand{\cenv}{\mathrm{C}^*_{\textup{env}}}

\newcommand{\nor}[1]{\left\Vert #1\right\Vert}       
\newcommand{\sca}[1]{\left\langle#1\right\rangle}  

\begin{document}

\title[Dynamical Systems and Cuntz families]{Representations of $\ca$-dynamical systems implemented by Cuntz families}

\author{Evgenios T.A. Kakariadis}
\address{Pure Mathematics Department, University of Waterloo,
   Ontario N2L3G1, Canada}
\email{ekakaria@uwaterloo.ca}

\author{Justin R. Peters}
\address{Department of Mathematics\\
   Iowa State University, Ames, Iowa, USA}
\email{peters@iastate.edu}

\thanks{2010 {\it  Mathematics Subject Classification.}
46L07, 47L25, 47L55, 47L65}
\thanks{{\it Key words and phrases:} dynamical systems, $\ca$-envelope, generalized crossed product, $\ca$-correspondences.}
\thanks{First author partially supported by the Fields Institute for Research in the Mathematical Sciences}
\thanks{The second author acknowledges partial support from the National Science Foundation, DMS-0750986}

\maketitle

\vspace{-.5cm}
\begin{center}
{\small \emph{Dedicated to the memory of William B. Arveson}}
\end{center}

\begin{abstract}
Given a dynamical system $(A,\al)$ where $A$ is a unital $\ca$-algebra and $\al$ is a (possibly non-unital) $*$-endomorphism of $A$, we examine families $(\pi,\{T_i\})$ such that $\pi$ is a representation of $A$, $\{T_i\}$ is a Toeplitz-Cuntz family and a covariance relation holds. We compute a variety of non-selfadjoint operator algebras that depend on the choice of the covariance relation, along with the smallest $\ca$-algebra they generate, namely the $\ca$-envelope. We then relate each occurrence of the $\ca$-envelope to (a full corner of) an appropriate twisted crossed product. We provide a counterexample to show the extent of this variety. In the context of $\ca$-algebras, these results can be interpreted as analogues of Stacey's famous result, for non-automorphic systems and $n>1$.

Our study involves also the one variable generalized crossed products of Stacey and Exel. In particular, we refine a result that appears in the pioneering paper of Exel on (what is now known as) Exel systems.
\end{abstract}

\section*{Introduction}

For this paper, a dynamical system $(A,\al)$ consists of a unital $\ca$-algebra $A$ and a (possibly non unital) $\al \in \End(A)$. We study the representation theory of universal objects that satisfy (more or less) the covariance relation
\[
\pi(\al(x)) = \sum_{i=1}^n T_i \pi(x) T_i^*, \foral x\in A,
\]
where $\{T_i\}$ is a Toeplitz-Cuntz family and $n \geq 1$. We were initially motivated by \cite{Cun77, Cun81, CouMuhSch10, Lac93, Mur02, Sta93}. Stacey's \emph{multiplicity $n$ crossed product} is the universal $\ca$-algebra $A \times_\al^n \bbN$, when $\pi$ is additionally assumed to be non-degenerate \cite{Sta93}, and a family of such representations is inherited by Cuntz's \emph{twisted crossed product} $A_\infty \rtimes_{\al_\infty} \sO_n$ \cite{Cun81}. Therefore a canonical $*$-epimorphism $A\times_\al^n \bbN \rightarrow p(A_\infty \rtimes_{\al_\infty} \sO_n)p$ is induced, where $p=[1_A] \in A_\infty$. Stacey shows that this $*$-epimorphism is injective when $n=1$ or $\al$ is an automorphism \cite[Propositions 3.3 and 3.4]{Sta93}, and one can ask for an extension of this result. The counterexample provided in Subsection \ref{S:counter}, shows that this is hopeless, and another approach should be considered.

This approach is given in the context of operator algebras both selfadjoint and non-selfadjoint, influenced by the ideas of Arveson \cite{Ar69, Ar08} and the work of many researchers in dilation theory (see \cite{BleMuPau00, DavKatsMem, DrMc05, Dun08, Ham79, KatsKribs06, Lac00, MS}, to mention but a few). Instead of examining \emph{a large} universal $\ca$-algebra $\sT_\sS$ generated by a system of relations $\sS$, one considers the non-involutive analogue $\sA_\sS$ subject to $\sS$. The goal is to associate \emph{the smallest} $\ca$-algebra $\cenv(\sA_\sS)$ that is generated by $\sA_\sS$, namely the $\ca$-envelope of $\sA_\sS$, to $\ca$-algebras generated by an invertible system $\sS'$ that dilates $\sS$. We coin this as \emph{Arveson's program on the $\ca$-envelope}\footnote{\ We emphasize that the term $\ca$-envelope must not be confused with the notion of an \emph{enveloping $\ca$-algebra}. The first one is a minimal $\ca$-algebra associated to $\sA_\sS$, whereas the second one is the maximal $\ca$-algebra $\sT_\sS$. See Subsection \ref{Ss:opalg} for the pertinent definitions.}.

There are some interesting consequences derived from this approach. First of all the $\ca$-envelope possesses a ring theoretical injective property. Thus in the following short exact sequence
\[
0 \longrightarrow \sJ \longrightarrow \sT_\sS \longrightarrow \cenv(\sA_\sS) \longrightarrow 0,
\]
the algebra $\sA_\sS$ is preserved isometrically by the unique $*$-epimorphism $\sT_\sS \rightarrow \cenv(\sA_\sS)$. The same scheme holds by substituting $\sT_\sS$ with any $\ca$-algebra generated by an isometric copy of $\sA_\sS$. Taking into account that $\sS$ is reflected isometrically inside $\sT_\sS$, hence in $\sA_\sS$, this gives a picture on how narrow one can be on defining universal $\ca$-algebras related to $\sS$. One cannot go beyond the minimal $\cenv(\sA_\sS)$. Additionally, the ideal $\sJ$ in $\sT_\sS$, namely \emph{the \v{S}ilov ideal} of $\sA_\sS$, plays the role of the non-commutative \v{S}ilov boundary and it is equipped with the analogous properties.

The motivating and intriguing part of this theory is that, even though $\cenv(\sA_\sS)$ always exists as a minimal object \cite{Ham79, DrMc05, Ar06, Kak11-2}, the interest lies exactly on connecting it to a natural $\ca$-algebra of invertible $\sS$-like relations. Nevertheless, the possibility of $\cenv(\sA_\sS)$ being $\sT_\sS$ is not excluded. In this paper we are interested in both such cases for a fixed system $\sS$.

To this end we define the non-selfadjoint operator algebras $A_{\textup{nd}} \times_\al^t \sT_n^+$ and $A_{\textup{nd}} \times_\al \sT_n^+$ in analogy with  Cuntz's twisted crossed product and Stacey's crossed product respectively, where the use of the superscript ``$t$'' is self-explanatory. In Theorem \ref{T:twi spr 2} we show that $\cenv(A_{\textup{nd}} \times_\al^t \sT_n^+)$ is a full corner of the twisted crossed product $A_\infty \rtimes_{\al_\infty} \sO_n$. This is the analogue of \cite[Propositions 3.3 and 3.4]{Sta93} for non-automorphic systems and $n>1$. On the other hand $\cenv(A_{\text{nd}} \times_\al \sT_n^+)$ is simply $A \times_\al ^n \bbN$ when $\al$ is unital (see Theorem \ref{T:env *-auto}). In fact the given counterexample shows that $A_{\text{nd}} \times_\al^t \sT_n^+$ and $A_{\text{nd}} \times_\al \sT_n^+$ are not canonically isomorphic unless $\al$ is an automorphism. From one point of view \cite{KPTT, OikPis99}, $A_{\text{nd}} \times_\al \sT_n^+$ resembles a \emph{maximal twisted product}, in contrast to $A_{\text{nd}} \times_\al^t \sT_n^+$ which resembles a \emph{minimal twisted product}; they coincide when $\al$ is an automorphism, but not in general.

The notation ${A}_\text{nd}$ emphasizes on that $A$ is represented non-degenerately inside the universal objects. One can define $A \times_\al^t \sT_n^+$ and/or $A\times_\al \sT_n^+$, without imposing this requirement. By appropriately weighting the generators, we show that $A_{\text{nd}} \times_\al^t \sT_n^+$ embeds in $A\times_\al^t\sT_n^+$, and in Theorem \ref{T:twi spr} we give the connection of $\cenv(A \times_\al^t \sT_n^+)$ with the twisted crossed product. This is the stepping stone to get the results explained in the previous paragraph.

We give applications of our results in cases that involve \emph{a transfer operator} $L$ following \cite[Section 4]{Exe03} and \cite{CouMuhSch10}. To do so we use the language of $\ca$-correspondences. In Theorem \ref{T:tens} we show that the Cuntz-Pimsner algebra of the correspondence $\sM_L\otimes X_n$ is a full corner of the twisted crossed product and has a rather special form. Here $\sM_L$ is Exel's correspondence \cite{Exe03} and $X_n$ is the Hawaiian earring graph on $n$ edges.

\emph{Exel systems} $(A,\al,L)$ were introduced in the pioneering paper of Exel \cite{Exe03} and have been under considerable investigation in a series of papers \cite{BroRae04, BroRaeVit09, HueRae11, Lar10}. In these papers a similar representation theory is used, when $n=1$. For this reason, we present the one-variable case separately, as it may be of independent interest. This allows us also to simplify proofs for $n>1$. A third reason for following this presentation of our results is that we correct an error in \cite[Theorem 4.7]{Exe03}: it is claimed that, when $\al$ is injective and $\al(A)$ is hereditary, then $A\rtimes_{\al,L} \bbN$ \cite[Definition 3.7]{Exe03} is $*$-isomorphic to $\fA(A,\al)$\cite[Definition 4.4]{Exe03}. However, in Theorem \ref{T:Exel} we show that $A\rtimes_{\al,L} \bbN$ is actually $*$-isomorphic to $A \times_\al^1 \bbN$, while in Proposition \ref{P:criterion} we show that $\fA(A,\al)$ coincides with
$A \times_{\al}^1 \bbN$ exactly when $\al(1_A)=1_A$. A fourth important reason is that $A_{\text{nd}} \times_\al \sT_n^+$ behaves in a different way when $n=1$, since $A_{\text{nd}} \times_\al \sT_1^+ \simeq A_{\text{nd}} \times_\al^t \sT_1^+$ (cf. Theorem \ref{T: not unit}, Theorem \ref{T: not unit multi}, Subsection \ref{S:counter}).

In Proposition \ref{P: Kats ideal} we give necessary and sufficient conditions for $A$ to be embedded injectively in $A\rtimes_{\al,L} \bbN$, adding to results obtained by Brownlowe and Raeburn \cite{BroRae04}. In fact, when $\al$ is unital or $(A,\al,L)$ is as in \cite[Section 4]{Exe03}, then $A\rtimes_{\al,L} \bbN$ is the Cuntz-Pimsner algebra of $\sM_L$, which we prove is a Hilbert bimodule (see Proposition \ref{P:unit simpli} and Theorem \ref{T:Exel}). This is quite surprising.
We remark that Brownlowe, Raeburn and Vitadello \cite{BroRaeVit09} examine Exel systems arising from commutative $\ca$-algebras, for which they show that $A\rtimes_{\al,L}\bbN$ is again the Cuntz-Pimsner algebra $\sO_{\sM_L}$. We wonder whether this is true in general; yet this lies outside the purview of this paper and is to be pursued elsewhere.

The paper is divided into three parts. In Section \ref{S:nota} we present notation and give the constructions that are used in the sequel. In Section \ref{S:one} we present the results concerning the one variable case. In Section \ref{S:multi} we proceed to the examination of the multivariable case. We conclude by giving the counterexample mentioned earlier.

We emphasize that in general the dynamical systems $(A, \al)$ are neither injective nor unital. Non-injectivity is treated by combining \cite{Pet84, Sta93, Kak11}: the objects related to $(A,\al)$ are identified with the objects related to the injective system $(A/\sR_\al, \dot\al)$. (Therefore the copy of $A$ inside these objects is $A/\sR_\al$, and it is $A$ exactly when $\al$ is injective.)
On the other hand, the trick of ``weighting'' the generators makes dilation theory far more tractable for our purposes and helps us treat the non-unital cases. We hope that the reader will discover applications of these tricks to other constructions.

\begin{notation}
The non-involutive part $\sT_n^+$ of the Toeplitz-Cuntz algebra $\sT_n$ is the linear span of elements $s_\mu$, for a word $\mu$ in $\bbF_n^+$. Whereas it would be convenient to use the semigroup $\bbF_n^+$ (resp. $\bbN$) in our notations when possible, we avoid doing so since the group $\bbF_n$ (resp. $\bbZ$) generated by $\bbF_n^+$ (resp. $\bbN$) does not contain the analogous relations that hold in $\sT_n$. Therefore $\sT_n^+$ and $\bbF_n^+$ differ based on their representation theory and we want to emphasize that distinction in our notation. (Note that in this way we deliberately avoid any confusion with the notation of \cite{DavKatsMem} and \cite{Li}.) The isometric copy of $\sT_n^+$ in $\sO_n$ will be denoted by $\sO_n^+$.

We will say that a family $\{T_i\}_{i=1}^n$ of operators in a $\sB(H)$ is a \emph{Toeplitz-Cuntz family} (resp. a \emph{Cuntz family}) if the $T_i$ are isometries with orthogonal ranges and $\sum_{i=1}^n T_i \leq I_H$ (resp. $\sum_{i=1}^n T_i =I_H$). For sake of simplicity we will often write $T_i \in \sT_n^+$ (resp. $T_ i \in \sO_n^+$), meaning that $\{T_i\}_{i=1}^n$ is a Toeplitz-Cuntz family (resp. a Cuntz family).

Throughout the paper a number of different operator algebras are defined, and for this reason we include an Index at the end of this paper. However, even if the established notation may not be self-contained, we do not proceed to change it. For example, Stacey's crossed product $A \times_\al^n \bbN$ could be denoted alternatively by $A_{\textup{nd}} \times_\al \sT_n$, as follows by its representation theory.

Finally, we make the convention that a \emph{canonical representation} is a representation that maps generators to generators of the same index.
\end{notation}

\section{Operator Algebras and Constructions} \label{S:nota}

\subsection{Operator Algebras and the $\ca$-envelope}\label{Ss:opalg}

An operator algebra $\sA$ is a closed subalgebra of some $\sB(H)$, where $H$ is a Hilbert space. When $\sA$ is closed under the involution inherited by $\sB(H)$ then it is a $\ca$-algebra. The representation theory of an operator algebra consists of \emph{completely contractive $($resp. isometric$)$ homomorphisms} $\rho \colon \sA \rightarrow \sB(H)$, i.e., every homomorphism $\rho_k \equiv \id_k \otimes \rho \colon M_k(\sA) \rightarrow \sB(H^{(k)})$ is contractive (resp. isometric).

A representation $\nu \colon \sA \rightarrow \sB(K)$ is called a \emph{dilation} of $\rho$ if $\rho(a)=P_H \nu(a)|_H$ for all $a\in \sA$. A complete isometry $\rho$ is called \emph{maximal} if any dilation $\nu$ of $\rho$ is trivial. By \cite{DrMc05} every complete isometry has a dilation that is maximal. Therefore $\rho$ is maximal if and only if, given a maximal dilation $\nu$ of $\rho$, there is a representation $\si$ such that $\nu= \rho \oplus \si$.

A pair $(\sC, \rho)$ such that $\sC$ is a $\ca$-algebra, $\rho\colon \sA \rightarrow \sC$ is a complete isometry and $\sC= \ca(\rho(\sA))$, is called a \emph{$\ca$-cover of $\sA$}. An ideal $J$ in $\sC$ is called \emph{boundary} if the restriction of the quotient map $q_J$ to $\rho(\sA)$ is a complete isometry; \emph{the \v{S}ilov ideal} $\sJ$ is the largest boundary ideal. Consequently, the pair $(\sC/\sJ, q_\sJ \circ \rho)$, called \emph{the $\ca$-envelope of $\sA$} and denoted simply by $\cenv(\sA)$, is the smallest $\ca$-cover generated by $\sA$. (It is clear that $\sC/\sJ$ contains no non-trivial boundary ideals.) In fact the $\ca$-envelope has a universal property: for every $\ca$-cover $(\sC,\rho)$ of $\sA$ there is a (necessarily) unique $*$- epimorphism $\Phi\colon \sC \rightarrow \cenv(\sA)$ such that $\Phi(\rho(a)) = a$ for all $a\in \sA$.

\begin{lemma}\label{L:fful_c*-env}
Let $\pi:\cenv(\sA) \rightarrow \sB(H)$ be a $*$-homomorphism. Then $\pi$ is faithful if and only if $\pi|_{\sA}$ is a complete isometry.
\end{lemma}
\begin{proof}
If $\pi$ is faithful then it is a complete isometry, hence $\pi|_\sA$ is a complete isometry. For the converse, the $*$-homomorphism $\widetilde{\pi}\colon \cenv(\sA)/\ker\pi \rightarrow \sB(H)$ is faithful. Hence $\ker\pi=(0)$, as a boundary ideal of $\sA$ in $\cenv(\sA)$.
\end{proof}

Unlike \cite{Ar08}, a representation $\pi$ of a $\ca$-cover $\sC$ of $\sA$ is called \emph{boundary} if the restriction of $\pi$ to $\sA$ is maximal \cite{DrMc05}. We say that $\sA$ has \emph{the unique extension property} if any faithful representation $\pi\colon \cenv(\sA) \rightarrow \sB(H)$ is a boundary representation \cite{Dun08}. In particular the free atomic representation of $\cenv(\sA)$ is boundary, therefore the irreducible representations of $\cenv(\sA)$ are boundary as direct summands of a boundary representation. Therefore, if $\sA$ has the unique extension property, then it admits a \emph{Choquet boundary} in the sense of \cite{Ar08}, i.e., the existence of sufficiently many irreducible boundary representations of $\cenv(\sA)$.

The existence of the $\ca$-envelope was first proved by Arveson in the case where there were enough boundary representations \cite{Ar69}. The first proof for the general case was given by Hamana \cite{Ham79}. Twenty five years later Dritschel and McCullough \cite{DrMc05} gave an independent proof for the existence of the $\ca$-envelope, simplified later by Arveson \cite{Ar06}. The first author gives an independent proof of Hamana's Theorem in \cite{Kak11-2}.
The existence of the Choquet boundary for separable operator systems (or operator algebras) was proved by Arveson \cite{Ar08} and it is still an open problem for the non-separable cases. Recently, Kleski \cite{Kle11} has proved that, for the separable case, the supremum can be replaced by the pointwise maximum.

We remark that the notion of the $\ca$-envelope should not be confused with the notion of the enveloping $\ca$-algebra of an involutive Banach algebra. The universal property of the $\ca$-envelope suggests that it is the smallest $\ca$-cover of $X$. One might want to rename the $\ca$-envelope as the \textit{$\ca$-minimal cover}, though it seems impractical to try to change established terminology.

Below we describe the construction of universal/enveloping (in general non-selfadjoint) operator algebras.

\subsection{Universal Operator Algebras}

An effective way to create a universal operator algebra $\sB$ with respect to generators and relations is, first to form the corresponding universal $\ca$-algebra $A$ in the sense of Blackadar \cite{Bla85}, and then consider the appropriate operator subalgebra $\sB$ of $A$. The universal representation of $\sB$ is then the restriction of the universal representation of $A$ to $\sB$.

A second way is to construct an operator algebra $\sB$ relative to (a certain family of) contractive representations of a Banach algebra $B$ with the property that every contractive homomorphism (in this family) of $B$ acting on a Hilbert space \emph{lifts} to a completely contractive homomorphism of $\sB$. (If in addition $B$ has an involution to take in account, then $\sB$ is the enveloping $\ca$-algebra of $B$.) Below is a brief discussion of the construction of such a universal object.

Let $B$ be a Banach algebra, $\sB$ be an operator algebra and $\phi \colon B \rightarrow \sB$ be a contractive algebraic homomorphism. We say that $\rho$ lifts to a completely contractive homomorphism $\widetilde{\rho}$ if the following diagram is commutative
\[
\xymatrix{
B \ar[r]^\phi \ar[dr]^\rho & \sB \ar@{-->}[d]^{\widetilde{\rho}} \\
& \sB(H)
}
\]

Suppose that the cardinality of $B$ is less or equal to a cardinal $\be$ that we choose such that $\be^{\aleph_0} = \be$. First let $\sF$ be the set of contractive representations $(H_i,\rho_i)$ of $B$ such that $\dim(H) \leq \be$. Let $\sH = \oplus_{i \in \sF} H_i$, $\phi := \oplus_{i \in \sF} \rho_i$, and $\sB: = \overline{\phi(B)}^{\sB(\sH)}$. Then $\sB$ is an operator algebra with the operator structure inherited from $\sB(\sH)$.  Equivalently, let the seminorms on $M_k(B)$
\begin{align*}
\om_k([a_{ij}])= \sup\{ \nor{[\rho(a_{ij})]}_{\sB(H_\rho^{(k)})} \mid
(H_\rho,\rho) \in \sF\}.
\end{align*}
If $\sN=\ker\om_1 (=\ker\phi)$, then $M_k(\sN)=\ker\om_k$. Hence the family of the induced norms $\nor{\cdot}_k$ on $M_k(B/\sN)$ with $\nor{[b_{ij}+\sN]}_k= \om_k([b_{ij}])$ is defined. Then $\sB$ is the completion of the image of $B / \sN$ in $\sB(\sH)$ (cf. \cite[Subsection 2.4.6]{BleLeM04}). Every $(H_i, \rho_i) \in \sF$ lifts to the completely contractive homomorphism $\widetilde{\rho_i}(\cdot) := P_{H_i} \cdot |_{H_i}$ of $\sB$. We will refer to $\id_{\sB} \equiv \oplus_{i \in \sF} \widetilde{\rho_i}$ as \emph{the universal representation of $\sB$} (the representations $(H_i, \widetilde{\rho_i})$ are the building blocks, as shown below).

Let $(H,\rho)$ be a contractive representation of $B$ where $H$ has arbitrary dimension, say $J_0 = \dim(H)$. Let the set $\sS$ consist of pairs $(J, \{K_j\}_{j \in J})$ such that $J \subseteq J_0$, the $K_j$ are mutually orthogonal subspaces of $H$ with $\dim(K_j) \leq \be$ and every $K_j$ is reducing for $\rho(B)$, i.e., every $K_j$ is $\rho(B)$-invariant and $\rho(B)^*$-invariant. The set $\sS$ is non-empty. Indeed, let $C$ be the $\ca$-algebra generated by $\rho(B)$ inside $\sB(H)$. Then the cardinality of $C$ is less or equal to $\be$, since it is the closure of the span of monomials of the form
\[
(\rho(a_1)^*)^{\epsilon_1}\rho(b_1) \rho(a_2)^* \cdots (\rho(a_n)^*)^{\epsilon_2}, \text{ with } \epsilon_1, \epsilon_2 = 0,1, \text{ and } a_l, b_l \in B.
\]
Then the Hilbert subspace $[C \xi]$, for $\xi \in H$, is reducing for $C$ (and consequently for $\rho(B)$), and has cardinality less or equal to $\be$. Choosing an $\eta \in H$ that is orthogonal to $[C\xi]$ we can form the subspace $[C \eta]$ that is orthogonal to $[C\xi]$ and so on.

We define a partial order in $\sS$ by the rule $(J, \{K_j\}_{j \in J}) \leq (J', \{K_{j'}\}_{j' \in J'})$, if $J \subseteq J'$ and $K_{j'} = K_{j}$ when $j'=j \in J$. If $\sI$ is a chain of such pairs then $(I, \{K_k\}_{k\in I})$,  with $I = \cup_{l \in \sI} J_l$, and $K_k = K_{j_l}$, when $k=j_l \in J_l$, is a maximal element of $\sI$ in $\sS$. Applying Zorn's Lemma we obtain a maximal element in $\sS$, say $(J,\{K_j\}_{j \in J})$.

If there is a non-zero $\xi \in H$ such that $\xi \perp K$, then $[C\xi]$ is reducing for $\rho(B)$ and has dimension less or equal that $\beta$, which leads to a contradiction. Thus $H = \oplus_{j \in J} K_j$. Therefore $\rho = \oplus_{j\in J} \rho_j$, where $\rho_j = \rho|_{K_j}$. But every $\rho_j$ is in $\sF$, hence lifts to a completely contractive representation of $\sB$, say $\widetilde{\rho_j}$. Thus $\widetilde{\rho} = \oplus_{j \in J} \widetilde{\rho_j}$ is a completely contractive representation of $\sB$, which is a lifting of $\rho$.

Similarly one can start with a family $\sF'$ of contractive representations of $B$. Starting with representations in $\sF'$ that act on Hilbert space with the upper bound on the dimension, one can construct an operator algebra $\fA(B,\sF')$ that has the universal property for representations in $\sF'$ (acting on Hilbert spaces of arbitrary dimension). We will refer to $\fA(B,\sF')$ as \emph{the universal operator algebra relative to $\sF'$}.

\subsection{Radical and Direct Limits}

Given a dynamical system $(A,\al)$, let the \emph{radical ideal} $\sR_\al$ be the closure of $\cup_n \ker\al^n$ \cite{Pet84}. Since  $x\in \sR_\al$ if and only if $\lim_n \al^n(x)=0$, an injective $*$-homomorphism is defined by
\begin{align*}
\dot\al: A/ \sR_\al \rightarrow A/\sR_\al: x + \sR_\al \mapsto
\al(x) + \sR_\al.
\end{align*}
The \emph{direct limit dynamical system $(A_\infty, \al_\infty)$ associated to $(A,\al)$} \cite{Sta93} is
\begin{align*}
 \xymatrix{
  A \ar[r]^{\al} \ar[d]^\al &
  A \ar[r]^{\al} \ar[d]^\al &
  A \ar[r]^{\al} \ar[d]^\al &
  \cdots \ar[r] &
  A_\infty \ar[d]^{\al_\infty} \\
  A \ar[r]^\al &
  A \ar[r]^\al &
  A \ar[r]^\al &
  \cdots \ar[r] &
  A_\infty
 }
\end{align*}
If $(A,\al)$ is unital, then $A_\infty \neq (0)$. It may be the case that $A_\infty$ is trivial, but $\al_\infty$ is always an automorphism of $A_\infty$. The image of $A$ in $A_\infty$ is $A/\sR_\al$, thus $(A_\infty, \al_\infty)$ coincides with the extension $\left( (A/\sR_\al)_\infty, (\dot{\al})_\infty \right)$ of $(A/\sR_\al, \dot\al)$.

\subsection{Exel systems}\label{Ss:transfer}

An \emph{Exel system} $(A,\al,L)$ consists of a dynamical system $(A,\al)$ with a \emph{transfer operator} $L\colon A \rightarrow A$, i.e., $L$ is a continuous positive linear map such that $L(\al(a)b)=aL(b)$, for all $a,b \in A$.

By definition, the range of $L$ is an ideal of $A$. By \cite[Proposition 2.3]{Exe03} $L$ is called \emph{non-degenerate} if one of the following equivalent conditions holds:
\begin{enumerate}
\item the mapping $\al\circ L$ is a conditional expectation onto $\al(A)$;
\item $\al\circ L \circ \al = \al$;
\item $\al(L(1))=\al(1)$.
\end{enumerate}
In that case $A= \ker\al \oplus \Im L$, as an orthogonal sum of ideals. When $L(1)=1$ then $L$ is non-degenerate and onto $A$. Thus $\al$ is injective.

\subsection{$\ca$-correspondences}

We use \cite{Kats04, Lan95} as a general reference. A \emph{$\ca$-correspondence $X$ over $A$} is a right Hilbert $A$-module together with a $*$-homomorphism $\phi_X\colon A \rightarrow \sL(X)$. A (\emph{Toeplitz}) \emph{representation} $(\pi,t)$ of $X$ consists of a $*$-homomorphism $\pi\colon A \rightarrow \sB(H)$ and a linear map $t\colon X \rightarrow \sB(H)$, such that $\pi(a)t(\xi)=t(\phi_X(a)(\xi))$ and $t(\xi)^*t(\eta)=\pi(\sca{\xi,\eta}_X)$, for all $a\in A$ and $\xi,\eta\in X$; by the $\ca$-identity we get also that $t(\xi)\pi(a)=t(\xi a)$. A representation $(\pi, t)$ is said to be \textit{injective} if $\pi$ is injective; then $t$ is an isometry.
The $\ca$-algebra generated by $(\pi,t)$ is the closed linear span of $t(\xi_1)\cdots t(\xi_n)t(\eta_m)^*\cdots t(\eta_1)^*$. Every pair $(\pi,t)$ defines a $*$-homomorphism $\psi_t\colon \sK(X)\rightarrow B$, such that $\psi_t(\Theta^X_{\xi,\eta})= t(\xi)t(\eta)^*$ \cite[Lemma 2.2]{KajPinWat98}.

Let $K$ be an ideal in $\phi_X^{-1}(\sK(X))$; we say that $(\pi,t)$ is \emph{$K$-coisometric} if $\psi_t(\phi_X(a))=\pi(a)$, for $a\in K$. Following \cite{Kats04}, the representations $(\pi,t)$ that are $J_{X}$-coisometric, where $J_X=\ker\phi_X^\bot \cap \phi_X^{-1}(\sK(X))$, are called \emph{covariant representations}. The ideal $J_X$ is the largest ideal on which the restriction of $\phi_X$ takes values in the compacts and is injective. When $\phi_X|_{J_X}$ induces an isomorphism onto $\sK(X)$ then $X$ is (called) \emph{a Hilbert bimodule} \cite{Kats03}.

The \emph{Toeplitz-Cuntz-Pimsner} algebra $\sT_X$ is the universal $\ca$-algebra for ``all'' representations of $X$, and the \emph{Cuntz-Pimsner algebra} $\sO_X$ is the universal $\ca$-algebra for ``all'' covariant representations of $X$. The \emph{tensor algebra} $\sT_{X}^+$ is the norm-closed algebra generated by the copies of $A$ and $X$ in $\sT_X$. For more details see \cite{FowRae99, Kats04, MS}.
There is an important connection between $\sT_X$ and $\sO_X$ established in full generality by Katsoulis and Kribs \cite[Theorem 3.7]{KatsKribs06}: the $\ca$-envelope of $\sT^+_X$ is $\sO_X$. As a consequence, the ideal $\sK(F(X)J_X)$ is the \v{S}ilov ideal of $\sT_X^+$.

For an ideal $K \subseteq \phi_X^{-1}(\sK(X))$ we denote by $\sO(K,X)$ the universal $\ca$-algebra for ``all'' $K$-coisometric representations of $X$. It is easy to deduce that $\sO(K,X) \simeq \sT_X/ \sI$, where $\sI$ is the ideal in $\sT_X$ generated by $\pi_u(a) - \psi_{t_u}(k)$, with $a\in K$, $\phi_X(a)=k \in \sK(X)$, and $(\pi_u,t_u)$ is the universal representation of $\sT_X$. The ideal $K \subseteq \phi^{-1}(\sK(X))$ may not be contained in $J_X$. Nevertheless, there are necessary and sufficient conditions that guarantee this. Note that $\sO_X$ is the minimal such $\ca$-algebra containing $A$.

\begin{lemma}\label{L:lemma J}
Let $X$ be a correspondence over $A$ and $K$ an ideal of $A$ contained in $\phi^{-1}(\sK(X))$. Then the following are equivalent:
\begin{enumerate}
\item $A \hookrightarrow \sO(K,X)$;
\item $\phi_X|_K$ is injective;
\item $K \subseteq J_X$;
\item every $J_X$-covariant representation is $K$-covariant;
\item $\sT_X^+ \hookrightarrow \sO(K,X)$.
\end{enumerate}
\end{lemma}
\begin{proof}
For $[(1) \Leftrightarrow (2)]$ see \cite[Lemma 2.2]{BroRae04}. If item $(2)$ above holds, then $K \subseteq \ker\phi_X^\perp$, hence $K \subseteq J_X$. The implications $[(3) \Rightarrow (4)]$ and $[(5) \Rightarrow (1)]$ are obvious. Finally assume that item $(4)$ holds. Then the diagram
\begin{align*}
\xymatrix{ \sT_X \ar[rr]^{q_1} \ar[dr]^{q_2} & & \sO_X\\
&  \sO(K,X) \ar[ur]^{q_3} &
}
\end{align*}
of canonical $*$-epimorphisms commutes. By \cite{KatsKribs06}, the restriction of $q_1$ to the tensor algebra is a complete isometry, hence $\sT_X^+ \hookrightarrow \sO(K,X)$.
\end{proof}

\section{One Variable Case}\label{S:one}

The reader is referred to \cite{BroRae04, BroRaeVit09, CouMuhSch10, Exe03, HueRae11, Lar10, Sta93} for examples that arise naturally in the context, that we won't repeat. We will write $1_A\equiv 1$.

\subsection{The Crossed Products $\fA(A,\al)$ and $A \times_\al^1\bbN$}\label{Ss:sta-exe}

Exel \cite{Exe03} defines the universal $\ca$-algebra $\fA(A,\al)$ relative to the class
\begin{align*}
\sF_E = \left\{ (\pi,T) \mid T \text{ isometry, and }\pi(\al(x))=T\pi(x)T^*, \foral x\in A\right\},
\end{align*}
that is generated by $\pi(A)$ and $T$. Since $T$ is an isometry, then
\[
T\pi(x)= \pi(\al(x))T, \foral x\in A.
\]
Moreover, $\pi(\al(1))= T\pi(1)T^*=\pi(\al(1))TT^*=TT^*\pi(\al(1))$.

A variation was given earlier by Stacey \cite{Sta93}:
for a dynamical system $(A,\al)$ with $A_\infty \neq (0)$, let $A \times_\al^1 \bbN$ be the universal $\ca$-algebra relative to the class
\begin{align*}
\sF_S = \left\{ (\pi,T) \in \sF_E \mid \pi \text{ is non-degenerate}\right\},
\end{align*}
generated by $\pi(x)T^n (T^*)^m$ for $n,m \in \bbZ_+$. Since $\pi$ is non-degenerate the $\ca$-algebra $A\times_\al^1 \bbN$ is generated by $\pi(A)$ and $T$.

Non-degeneracy of $A$ is not assumed in the definition of $\fA(A,\al)$, otherwise $\fA(A,\al)$ would coincide with $A \times_\al ^1 \bbN$. Thus $1\in A$ is a projection in $\fA(A,\al)$. The connection between $\fA(A,\al)$ and $A\times_\al^1 \bbN$ is established below.

\begin{proposition}\label{P:sta-exe}
The $\ca$-algebra $A\times_\al^1 \bbN$ is $*$-isomorphic to the $\ca$- subalgebra of $\fA(A,\al)$ generated by $\pi_u(A)$ and $T_u\pi_u(1)$, where $(\pi_u,T_u)$ is the universal representation of $\fA(A,\al)$. Moreover, $\pi_u(1)$ is a unit for $A\times_\al^1 \bbN$.
\end{proposition}
\begin{proof}
First note that $\pi_u(1)$ is a unit for $\ca(\pi_u(A), T_u\pi_u(1))$, since
\[
\pi_u(1)\cdot T_u \pi_u(1) = \pi_u(1) \pi_u(\al(1)) T_u = \pi_u(\al(1)) T_u = T_u \pi_u(1).
\]
Hence the restriction of the identity representation on $K=\pi_u(1)H_u$ defines a faithful representation of $\ca(\pi_u(A),T_u\pi_u(1))$. Moreover $\pi_u|_K$ is non-degenerate for $A$ and $T_u\pi_u(1)|_K$ is an isometry.
Therefore there is a canonical $*$-epimorphism $\Phi \colon A \times_\al^1 \bbN \rightarrow \ca(\pi_u(A), T_u\pi_u(1))$.

On the other hand, let $(\pi,T)$ be a pair for $A\times_\al^1 \bbN$, where $\pi$ is non-degenerate and $T$ is an isometry. Then $(\pi, T)$ defines a pair also for $\fA(A,\al)$. Thus there is a canonical $*$-epimorphism $\Psi \colon \fA(A,\al) \rightarrow A \times_\al^1 \bbN$.
Therefore $\Psi(T_u \pi_u(1))= T \pi(1) = T$, since $\pi$ is assumed non-degenerate. Thus the restriction of $\Psi$ to $\ca(\pi_u(A), T_u\pi_u(1))$ is still onto $A\times_\al ^1 \bbN$.

It is straightforward that $\Psi \circ \Phi = \id_{A \times_\al^1 \bbN}$, which completes the proof.
\end{proof}

By the following Proposition, we can restrict to injective dynamical systems. Recall that $\dot\al(x+ \sR_\al)= \al(x) + \sR_\al$.

\begin{proposition}\label{P:R_al for fA}
The $\ca$-algebras $\fA(A,\al)$ and $\fA(A/\sR_\al,\dot\al)$ are $*$- isomorphic; analogously, $A \times_\al ^1 \bbN \simeq (A/\sR_\al) \times_{\dot\al}^1 \bbN$.
Moreover, $A/\sR_\al$ embeds in $(A/\sR_\al) \times_{\dot\al}^1 \bbN$ and $\fA(A/\sR_\al,\dot \al)$.
\end{proposition}
\begin{proof}
The first part goes as in the proof of \cite[Proposition 3.14]{Kak11}; in short, if $x\in \ker\al^n$, then $T^n\pi(x)=\pi(\al^n(x))T^n=0$, hence $\pi(x)=0$ since $T$ is an isometry. A limit argument then shows that $\pi|_{\sR_\al}=0$; thus $(\pi\circ q_{\sR_\al},T)$ is a pair for $\fA(A,\al)$ if an only if $(\pi,T)$ is a pair for $\fA(A/\sR_\al,\dot\al)$.

For the second part, it suffices to prove it in the case where $\al$ is injective. To this end let $(\pi,U)$ be a representation of the usual crossed product $A_\infty \rtimes_{\al_\infty} \bbZ$. Then $(\pi,U)$ defines a pair for $\fA(A,\al)$. Hence there is a canonical $*$-epimorphism $\fA(A,\al) \rightarrow A_\infty \rtimes_{\al_\infty} \bbZ$, that fixes $A$. Thus we obtain a canonical factorization of $A \hookrightarrow A_\infty \hookrightarrow A_\infty \rtimes_{\al_\infty} \bbZ$ by
\[
A \rightarrow A \times_\al ^1 \bbN \subseteq  \fA(A,\al) \rightarrow A_\infty \rtimes_{\al_\infty} \bbZ.
\]
Therefore $A$ embeds in $A \times_\al ^1 \bbN$.
\end{proof}

\begin{proposition}\label{P:criterion}
Let $(A,\al)$ be a dynamical system with $A$ unital and $\al$ injective. Then the following are equivalent
\begin{enumerate}
\item $\al(1)=1$;
\item $A \times_\al^1 \bbN \simeq \fA(A,\al)$;
\item $1\in A$ is a unit for $\fA(A,\al)$.
\end{enumerate}
\end{proposition}
\begin{proof}
Let $(\pi_u,T_u)$ be the universal representation of $\fA(A,\al)$. If item $(1)$ above holds then $\pi_u(1)$ is a unit for $\fA(A,\al)$, hence $\pi_u$ can be assumed non-degenerate and item $(2)$ is implied.
It is obvious that item $(2)$ implies $(3)$.
To end the proof, assume that item $(3)$ holds. Then for $T_u \in \fA(A,\al)$ we obtain $\pi_u(1)T_u=T_u \pi_u(1)$. Since $T_u$ is an isometry, then
\begin{align*}
\pi_u(\al(1)) & =\pi_u(\al(1)) T_u^*T_u = T_u^*\pi_u(1) T_u
 = T_u^* T_u \pi_u(1) = \pi_u(1).
\end{align*}
Since $\al$ is injective, $\pi_u$ is faithful, hence $\al(1)=1$.
\end{proof}

\subsection{The Toeplitz algebra $\sT(A,\al,L) \simeq \sT_{\sM_L}$}\label{Ss:toeplitz Exel}

In his pioneering paper Exel \cite{Exe03} examines operator algebras related to a system $(A,\al,L)$. In his original definition, $\sT(A,\al,L)$ is the universal $\ca$-algebra relative to the class
\begin{align*}
\left\{ (\pi,S) \mid S\pi(x)=\pi(\al(x))S, \, \pi(L(x))=S^*\pi(x)S, \foral x\in A\right\},
\end{align*}
generated by $\pi(A)$ and $S$. A priori $\sT(A,\al,L)$ seems to differ from the universal $\ca$-algebra subject to the same representation theory, but generated by $\pi(A)$ and $\pi(A)S$. These two objects are equivalent as $\pi$ can be chosen to be unital \cite{BroRae04, Lar10}.

The existence of at least one such pair $(\pi,S)$ comes from a representation of the following $\ca$-correspondence. Let $A$ be the semi-$A$-inner product with
\begin{align*}
x\cdot y= x\al(y), \, \sca{x,y}=L(x^*y), \, \foral x,y \in A,
\end{align*}
and let $\sM_L$ be the Hilbert $A$-module associated to it. That is, $\sM_L$ is the completion of the quotient $A/\sN_1$, where $\sN_1= \{ x\in A: L(y^*x)=0, \foral y\in A\}$.
Along with the $*$-homomorphism $\phi_{\sM_L}\colon A \rightarrow \sL(\sM_L)$,
such that $\phi(y)(x+\sN_1)=yx + \sN_1$, the module $\sM_L$ becomes the $\ca$- correspondence over $A$ \emph{associated to $(A,\al,L)$}. By definition, every pair $(\pi,S)$ of $\sT(A,\al,L)$ defines a representation of $\sM_L$.

\begin{lemma}\label{L:quo}
If $\xi \in A$ then $\xi + \sN_1= \xi \al(1) + \sN_1$.
\end{lemma}
\begin{proof}
By the computation
\[
L( \eta^* (\xi - \xi \al(1))) = L(\eta^*\xi) - L(\eta^*\xi\al(1)) = L(\eta^*\xi)-L(\eta^*\xi)\cdot 1 =0,
\]
we obtain that $\xi - \xi \al(1) \in \sN_1$.
\end{proof}

Let $\sM_{L^n}$ be the $\ca$-correspondence associated to $(A,\al^n,L^n)$. For $n=0$, we identify $\sM_{L^0}$ with the trivial $\ca$-correspondence $A$ over $A$.
Let $\sM_\infty$ be the direct sum $\oplus_{n\geq 0} \sM_{L^{n}}$. For every $n\in \bbZ_+$, the mappings
\begin{align*}
\ga\colon \sM_{L^n} \rightarrow \sM_{L^{n+1}}: x + \sN_n \mapsto \al(x) + \sN_{n+1},
\end{align*}
are adjointable with $\ga_n^*(x+\sN_{n+1})=L(x) + \sN_n$.
Define the pair $(\rho,S)$ by
\begin{align*}
& S \colon \sM_\infty \rightarrow \sM_\infty: (x_0, x_1+\sN_1, \dots)\mapsto (0, \ga_0(x_0), \ga_1(x_1+\sN_1),\dots),\\
& \rho(y)\colon \sM_\infty \rightarrow \sM_\infty: (x_0,x_1+\sN_1,\dots) \mapsto (yx_0, yx_1+\sN_1,\dots).
\end{align*}
We record for further use that
\[
S^* \colon \sM_\infty \rightarrow \sM_\infty: (x_0, x_1+\sN_1, \dots)\mapsto (L(x_1), L(x_2)+\sN_1,\dots).
\]
Then the pair $(\rho,S)$ defines a pair for $\sT(A,\al,L)$. Moreover $\rho$ is injective and $\rho(1)$ is the identity of $\sL(\sM_\infty)$. Thus $\ca(\rho(A),\rho(A)S) = \ca(\rho(A),S)$.

The following is proved by Brownlowe and Raeburn \cite{BroRae04} and Larsen \cite{Lar10}, independently. We give an alternative short proof.

\begin{theorem}\label{T:*-iso 1} \cite{BroRae04, Lar10}
The $\ca$-algebras $\sT(A,\al,L), \sT_{\sM_L}$ and $\ca(\rho,S)$ are $*$-isomorphic. Consequently, $1 \in A$ acts as a unit on $\sT(A,\al,L)$.
\end{theorem}
\begin{proof}
By universality there are canonical $*$-epimorphisms $\Phi_1$ and $\Phi_2$ such that
\[
\sT_{\sM_L} \stackrel{\Phi_1}{\longrightarrow} \sT(A,\al,L) \stackrel{\Phi_2}{\longrightarrow}
\ca(\rho(A),\rho(A)S)= \ca(\rho,S),
\]
where we use that $\rho(1)=I$. Then the $*$-epimorphism $\Phi=\Phi_2 \circ \Phi_1$ defines a pair $(\rho,t)$ for $\sM_L$ with
\[
\rho(x)= \Phi(x), \, t(x+\sN_1):= \Phi(x+\sN_1) = \rho(x)S, \, \foral x\in A.
\]

Note that $\rho$ is injective and $(\rho,t)$ admits the gauge action $\beta_z=\ad_{u_z}$, for $z\in \bbT$, where $u_z(0,\dots,0,x+\sN_n,0,\dots) = z^n(0, \dots, 0, x + \sN_n,0 ,\dots)$. Moreover, for a compact operator $\Theta_{x+\sN_1, y + \sN_1} \in \sK(\sM_L)$ observe that
\begin{align*}
\psi_t(\Theta_{x+\sN_1, y+\sN_1})(a, 0, \dots)& = \rho(x)SS^*(ya,0,\dots)=0,
\end{align*}
for all $a \in A$. Hence $\psi_t(\sK(\sM_L))|_A =0$. Therefore, if there was an $a\in A$ such that $\rho(a)=\psi_t(k)$ for some $k \in \sK(\sM_L)$, then
\begin{align*}
(aa^*,0,\dots)=\rho(a)(a^*,0,\dots)= \psi_t(k)(a^*,0,\dots)=0,
\end{align*}
hence $a=0$. Thus the ideal $\{x\in A \mid \rho(x) \in \psi_t(\sK(\sM_L))\}$ is trivial. Then $\Phi$ is a $*$-isomorphism, in view of the gauge invariance theorem for Toeplitz-Cuntz-Pimsner algebras \cite{FowRae99, Kats04}. Consequently, $\Phi_1$ and $\Phi_2$ are $*$-isomorphisms.
\end{proof}

\subsection{Exel's Crossed Product $A\rtimes_{\al,L} \bbN \simeq \sO(K_\al,\sM_L)$}\label{Ss:crpr exel}

In what follows we identify $\sT(A,\al,L), \sT_{\sM_L}$ with $\ca(\rho,S)$ and we will omit the symbol $\rho$.

A \emph{redundancy} $(a,k)$ is a pair in $A\times \overline{ASS^*A}$ such that
\begin{align*}
abS=kbS, \foral b\in A.
\end{align*}
Equivalently, $\phi_{\sM_L}(a)= k \in \sK(\sM_L)$ and $\rho(a) t(b+ \sN_1) = \psi_t(k) t(b+\sN_1)$.
\emph{Exel's crossed product} $A\rtimes_{\al,L} \bbN$ is the quotient of $\sT(A,\al,L)$ by the ideal $\sI$ generated by $a-k$, where $(a,k)$ is a redundancy and $a\in \overline{A\al(A) A}$ \cite{Exe03}.

Brownlowe and Raeburn \cite[Corollary 3.6]{BroRae04} show that $A\rtimes_{\al,L} \bbN$ is $*$-isomorphic to the relative Cuntz-Pimsner algebra $\sO(K_\al,\sM_L)$ where
\[
K_\al:=\overline{A\al(A)A} \cap \phi_{\sM_L}^{-1}(\sK(\sM_L)).
\]
(cf. remarks preceding Lemma \ref{L:lemma J}). In \cite[Theorem 4.2]{BroRae04} it is proved that $A \hookrightarrow A\rtimes_{\al,L} \bbN$ if and only if $L$ is \emph{almost faithful on $K_\al$}, i.e., if $x\in K_\al$ and $L((xy)^*xy)=0$ for all $y\in A$, then $x=0$. This is equivalent to letting $\phi_{\sM_L}|_{K_\al}$ be injective. In view of Lemma \ref{L:lemma J}, which extends \cite[Lemma 2.2]{BroRae04}, we obtain a list of other sufficient and necessary conditions. We gather all these in the following Proposition.

\begin{proposition}\label{P: Kats ideal}
Let $\sM_L$ be the $\ca$-correspondence associated with an Exel system $(A,\al,L)$. Then the following are equivalent:
\begin{enumerate}
\item $A \hookrightarrow \sO(K_\al,\sM_L)$;
\item $\phi_{\sM_L}|_{K_\al}$ is injective;
\item $K_\al \subseteq J_{\sM_L}$;
\item there is a canonical $*$-epimorphism $\sO(K_\al,\sM_L) \rightarrow \sO_{\sM_L}$;
\item $\sT^+_{\sM_L} \hookrightarrow \sO(K_\al,\sM_L)$;
\item $L$ is almost faithful on $K_\al$.
\end{enumerate}
\end{proposition}

In particular, for unital systems the picture is further simplified.

\begin{proposition}\label{P:unit simpli}
Let $(A,\al,L)$ be an Exel system such that $\al(1)=1$. Then the following are equivalent:
\begin{enumerate}
\item $A \hookrightarrow \sO(K_\al, \sM_L)$;
\item $\sO(K_\al,\sM_L) \simeq \sO_{\sM_L}$.
\end{enumerate}
If any of the above holds, then $\sM_L$ is a Hilbert bimodule.
\end{proposition}
\begin{proof}
Since $A \hookrightarrow \sO_{\sM_L}$ it suffices to show $[(1) \Rightarrow (2)]$. By $\al(1)=1$ we obtain $\overline{A \al(A) A} = A$. Hence $K_\al = \phi^{-1}(\sK(\sM_L))$, thus $J_{\sM_L} \subseteq K_\al$. If item $(1)$ holds then $K_\al \subseteq J_{\sM_L}$, therefore $J_{\sM_L}= K_\al = \phi^{-1}(\sK(\sM_L))$.
\end{proof}

\begin{remark}\label{R:corr Exel}
In \cite[Section 4]{Exe03} Exel shows that for the systems $(A,\al,L)$ that arise (uniquely) by a $*$-monomorphism with hereditary range, the $\ca$-algebras $\fA(A,\al)$ and $A\rtimes_{\al,L} \bbN$ coincide. But $A\rtimes_{\al,L} \bbN$ is a quotient of the unital $\ca$-algebra $\sT(A,\al,L)$, hence $\fA(A,\al)$ is unital. In view of Proposition \ref{P:criterion} this implies that $\al$ is unital, which is not the general case as stated in \cite[Theorem 4.7]{Exe03}. The error appears in the computations of the proof of \cite[Proposition 4.6]{Exe03}, where $1\in A$ is treated as a unit of $\fA(A,\al)$.
\end{remark}

\subsection{Revisiting Exel's Example}\label{S:Ex}

We correct the error that appears in \cite[Theorem 4.7]{Exe03}. In fact we do more; by following an alternate route, we shed light also on the non-injective case (see Theorem \ref{T:main 2}). To simplify notation, we identify $A \rtimes_{\al,L} \bbN$ with $\sO(K_\al,\sM_L)$, and $\sT(A,\al,L), \sT_{\sM_L}$ with $\ca(\rho,S)$, where we will omit writing $\rho$.

Our basic assumption is that $(A,\al,L)$ satisfies
\begin{align*}
(\dagger)\qquad \al\circ L (x) = \al(1)x\al(1), \foral x\in A.
\end{align*}
In this case, $\al(L(1))=\al(1)$ and $L$ is \emph{non-degenerate} in the sense of \cite{Exe03}.

The choice of $(\dagger)$ enables us to identify $K_\al$ with $\overline{A \al(A) A}$.

\begin{proposition}\label{P:simple}
Let an Exel system $(A,\al,L)$ that satisfies $(\dagger)$. Then we obtain $\phi^{-1}_{\sM_L}(\sK_{\sM_L}) = \overline{A\al(A) A} + \ker\phi_{\sM_L}$. Therefore $J_{\sM_L} \subseteq K_\al = \overline{A \al(A) A} $.
\end{proposition}
\begin{proof}
Trivially $\ker\phi_{\sM_L} \subseteq \phi^{-1}_{\sM_L}(\sK_{\sM_L})$. For all $x, \xi \in A$, we obtain
\begin{align*}
\Theta_{\al(x) + \sN_1, 1+\sN_1} (\xi + \sN_1)
& = \al(x) \al(L(\xi)) + \sN_1
  = \al(x) \xi \al(1) + \sN_1\\
& = \al(x)\xi + \sN_1 = \phi_{\sM_L}(\al(x)) (\xi + \sN_1),
\end{align*}
where we have used Lemma \ref{L:quo}. Therefore $\phi_{\sM_L}(\al(A)) \subseteq \sK(\sM_L)$, hence $\overline{A\al(A) A} \subseteq \phi^{-1}_{\sM_L}(\sK(\sM_L))$. Now, let $\xi, \eta, \zeta \in A$; then
\begin{align*}
\Theta_{\eta + \sN_1, \zeta+\sN_1} (\xi + \sN_1)
& = \eta \al(L(\zeta^*\xi)) + \sN_1
   = \eta \al(1) \zeta^* \xi \al(1) + \sN_1\\
& = \eta \al(1) \zeta^* \xi + \sN_1
   = \phi_{\sM_L}( \eta \al(1) \zeta^*) (\xi + \sN_1).
\end{align*}
By passing to limits of linear combinations, we obtain that for every $k\in \sK(\sM_L)$ there is an $x \in \overline{A \al(A) A}$ such that $\phi_{\sM_L}(x)=k$. Thus, $\phi^{-1}_{\sM_L}(\sK_{\sM_L}) \subseteq \overline{A \al(A) A} + \ker\phi_{\sM_L}$.
To end the proof, recall that $J_{\sM_L} = \phi^{-1}_{\sM_L}(\sK(\sM_L)) \cap \ker\phi_{\sM_L}^\perp$. Therefore
\begin{align*}
J_{\sM_L}
& = \phi^{-1}(\sK(\sM_L)) \cap \ker\phi_{\sM_L}^\perp \cap (\overline{A \al(A) A}  + \ker\phi_{\sM_L})\\
& =\phi^{-1}(\sK(\sM_L)) \cap \ker\phi_{\sM_L}^\perp \cap \overline{A \al(A) A}
   \subseteq K_\al,
\end{align*}
and $\overline{A\al(A)A} \subseteq \overline{A\al(A)A}\cap (\overline{A\al(A)A} + \ker\phi_{\sM_L}) = K_\al \subseteq \overline{A\al(A)A}$.
\end{proof}

\begin{theorem}\label{T:Exel}
Let an Exel system $(A,\al,L)$ that satisfies $(\dagger)$. Then the following are equivalent
\begin{enumerate}
\item $A \hookrightarrow \sO(K_\al,\sM_L)$;
\item $\sO(K_\al,\sM_L) \simeq \sO_{\sM_L}$.
\end{enumerate}
If any of the above holds, then $\sM_L$ is a Hilbert bimodule.
\end{theorem}
\begin{proof}
In view of Propositions \ref{P: Kats ideal} and \ref{P:simple}, if any of items $(1)$ or $(2)$ holds, then $J_{\sM_L} =K_\al = \overline{A\al(A)A}$ which implies that they are equivalent. Then $\phi_{\sM_L}|_{J_{\sM_L}}$ is onto $\sK(\sM_L)$, hence $\sM_L$ is a Hilbert bimodule.
\end{proof}

We want to examine the connection with $A\times_\al^1 \bbN$. Towards this, the first step is the following Proposition.

\begin{proposition}\label{P:com diag}
Let $(A,\al,L)$ be an Exel system that satisfies $(\dagger)$. Then there are canonical $*$-epimorphisms such that the diagram
\begin{align*}
\xymatrix{ \sT(A,\al,L) \ar[dr]^{\tilde{\phi}} \ar[rr]^\phi & & \sO(K_\al,\sM_L)  \ar[dl]^{\hat\phi} \\
           & A \times_\al^1 \bbN&
}
\end{align*}
commutes.
\end{proposition}
\begin{proof}
Recall that $A \times_\al^ 1 \bbN \hookrightarrow \fA(A,\al)$ by Proposition \ref{P:sta-exe}. Let $(\pi,T)$ be a pair for $\fA(A,\al)$ and define $S'=\pi(\al(1))T=T\pi(1)$. Then,
\begin{align*}
S'\pi(x)=T\pi(x) =\pi(\al(x))T=\pi(\al(x))S'
\end{align*}
for all $x\in A$. Moreover,
\begin{align*}
(S')^*\pi(x)S' & = T^* \pi(\al(1)) \pi(x) \pi(\al(1)) T = T^* \pi(\al(1)x\al(1)) T\\
&= T^* \pi(\al\circ L(x)) T= T^*T \pi(L(x))= \pi(L(x)),
\end{align*}
for all $x\in A$. Therefore $(\pi,S')$ defines a representation of $\sT(A,\al,L)$, thus the $*$-epimorphism $\tilde\phi$ is defined by the universal property. The rest of the proof follows similar arguments with those of \cite[Theorem 4.7]{Exe03}.
\end{proof}

Given $(A,\al,L)$ satisfying $(\dagger)$, one can define a new system $(A/\sR_\al, \dot{\al}, \dot{L})$ satisfying $(\dagger)$, such that $\dot\al$ is injective and $\dot{L}$ is a well defined (positive) operator. Indeed, for $x\in \sR(A)$ we obtain
\[
\nor{\al^{n+1}(L(x))} = \nor{\al^n( \al(1) x \al(1))}  = \nor{\al^{n+1}(1) \al^n(x) \al^{n+1}(1)} \leq \nor{\al^{n}(x)},
\]
thus $L(\sR_\al) \subseteq \sR_\al$ and $\dot{L}$ is well defined.

\begin{theorem}\label{T:main 2}
Let $(A,\al,L)$ be an Exel system that satisfies $(\dagger)$. Then the following diagram commutes
\begin{align*}
\xymatrix{ \sT(A,\al,L) \ar[dd] \ar[dr] \ar[rr] & &  \sO(\sK_\al,\sM_L) \ar[dl] \ar[dd] \\
           & A\times_\al^1 \bbN \simeq (A/\sR_\al) \times_{\dot\al}^1 \bbN \ar[dr]^{\simeq} &  \\
           \sT(A/\sR_\al,\dot\al,\dot{L})  \ar[rr] \ar[ur]  & & \sO(\sK_{\dot{\al}},\sM_{\dot{L}}) }
\end{align*}
and consequently, $\sM_{\dot{L}}$ is a Hilbert bimodule.
\end{theorem}
\begin{proof}
In view of the previous remarks, Propositions \ref{P:R_al for fA} and \ref{P:com diag}, and Theorem \ref{T:Exel}, it suffices to show that $A\times_\al^1 \bbN \simeq \sO(K_\al,\sM_L)$ when $\al$ is injective. In this case, relation $(\dagger)$ for $x=1$ gives that $L(1)=1$. Therefore the operator $S$ of Theorem \ref{T:*-iso 1} becomes an isometry.

Let $(\pi,S')$ be the induced representation of $\sO(K_\al,\sM_L)$ by $(\rho,S)$. Then $S'$ is an isometry, as the range of the isometry $S$. By Proposition \ref{P:simple}, we obtain that $(\rho(\al(x)), S\rho(x)S^*)$ is a redundancy, hence $\pi(\al(x)) = S' \pi(x) S'^*$ for all $x\in A$. Thus $(\pi,S')$ defines a representation for $\sO(K_\al,\sM_L)$, hence there is a canonical $*$-epimorphism $\sO(K_\al,\sM_L) \rightarrow A\times_{\al}^1\bbN$, that is the inverse for $\hat\phi$ of Proposition \ref{P:com diag}.
\end{proof}

\begin{remark}
Let the sets $Q=\{ [\text{relation } (\dagger)], [\al(A) \text{ is hereditary}]\}$ and $W=\{[L(1)=1], [\al \text{ is injective}]\}$. It is easy to check that whenever an item from $Q$ and an item from $W$ hold, then all items in $Q$ and $W$ hold. The implication that concerns us at this point is that, when $\al(A)$ is hereditary and $\al$ is injective then the system $(A,\al,L)$ satisfies $(\dagger)$ and $L(1)=1$. This follows readily by the remarks before \cite[Proposition 4.3]{Exe03}, which completes the correction of \cite[Theorem 4.7]{Exe03} mentioned earlier. In any of these cases, the transfer operator is uniquely determined, since $\al$ is injective.

Moreover, we have showed that in this case there is a canonical $*$- isomorphism $\sO_{\sM_L} \rightarrow A \times_\al^1 \bbN$. Surprisingly, the converse also holds and we leave the proof of this metamathematical remark to the reader.
\end{remark}

\begin{remark}
The systems that we examine may not be unital. In fact this case is trivial: when $1=\al(1)=L(1)$ and $(\dagger)$ holds, then $\al$ is injective and $\al\circ L = \id$; therefore $\al$ is an automorphism with $\al^{-1}=L$. Hence $\sT(A,\al,L)$ is the $\ca$-algebra generated by the universal representation of $\fA(A,\al, \text{is})_r$ \cite{Kak11}, and $A\times_{\al,L} \bbN$ is the usual crossed product $A\rtimes_\al \bbZ$.
\end{remark}

\begin{remark}
Brownlowe, Raeburn and Vitadello \cite{BroRaeVit09} examine a second class of examples of Exel systems induced by classical dynamical systems $(T,\phi)$. For this class they show that the set of redundancies is the Katsura ideal, moreover that $C_0(T)\times_{\al,L}\bbN$ is again the Cuntz-Pimsner algebra $\sO_{\sM_L}$. This identification is obtained in our case and we wonder whether this is true in general.
\end{remark}

\subsection{$\ca$-envelopes}

In this subsection we examine the $\ca$-envelope of the non-selfadjoint parts of $\fA(A,\al)$ and $A \times_\al^1 \bbN$, and the form of $\sO_{\sM_L}$.

\subsubsection{The non-selfadjoint part of $\fA(A,\al)$ and $A \times_\al^1 \bbN$}

We remark that the automorphic extension of the (unital) automorphism $\al_\infty$ of $A_\infty$ to the multiplier algebra $\sM(A_\infty)$ will be denoted by the same symbol.

\begin{theorem}\label{T: unit}
Let $A \times_\al \sT_1^+ := \alg\{\pi_u(A), T_u\} \subseteq \fA(A,\al)$. Then the $\ca$-envelope of $A \times_\al \sT_1^+$ is the $\ca$-subalgebra $\ca(B, U)$ of $\sM(B) \rtimes_{\be} \bbZ$, where $B=(A/\sR_\al)_\infty$ and $\be=(\dot\al)_\infty$.
\end{theorem}
\begin{proof}
It suffices to show it for $\al$ injective, i.e., $\sR_\al=(0)$. By \cite{Sta93} the pair $(\pi_u,T_u)$ extends to a covariant unitary pair $(\pi,U)$ of $(A_\infty, \al_\infty)$. In turn a pair $(\pi,U)$ defines a pair for $\fA(A,\al)$. Therefore the canonical embedding $\fA(A,\al) \rightarrow \sM(A_\infty) \rtimes_{\al_\infty} \bbZ$ defines a completely isometric homomorphism of $A \times_\al \sT_1^+$. It is immediate that $\ca(\pi(A),U) = \ca(\pi(A_\infty), U)$ since the powers $U^{-n} \pi(x) U^n$ generate $\pi(A_\infty)$. Thus $\ca(\pi(A_\infty), U)$ is a $\ca$-cover of $A \times_\al \sT_1^+$.

Let $\{\beta_z\}$ be the gauge action defined on $\sM(A_\infty) \rtimes_{\al_\infty} \bbZ$; then $A \times_\al \sT_1^+$ is $\beta_z$-invariant. Thus, if $\sJ$ is the \v{S}ilov ideal in $\ca(\pi(A_\infty), U)$ and it is non-trivial then it is also $\beta_z$-invariant. Therefore $\sJ$ intersects non trivially the fixed point algebra $\ca(A_\infty, I)$ of $\ca(\pi(A_\infty), U)$. Since $A_\infty$ is essential in $\sM(A_\infty)$, then $\sJ$ intersects $A_\infty$. Since $A_\infty$ is a direct limit there is an $x\in A_\infty$ such that $x\in \sJ$ and $\al_\infty^n(x) \in A$, for some $n$. But then $\pi(\al_\infty^n(x)) = U^n \pi(x) U^{-n} \in \sJ$. Therefore $\sJ$ intersects with $A$, hence with $A \times_\al \sT_1^+$, which is a contradiction.
\end{proof}

Recall that if $p$ is a projection in a $\ca$-algebra $\fA$ then the corner $p\fA p$ is called \emph{full} if the only ideal of $\fA$ containing $p\fA p$ is $\fA$.

\begin{theorem}\label{T: not unit}
Let $A_{\textup{nd}} \times_\al \sT_1^+:= \alg\{\pi_u(A), T_u\pi_u(1)\} \subseteq A \times_\al^1 \bbN$. Then the $\ca$-envelope of $A_{\textup{nd}} \times_\al \sT_1^+$ is a full corner of $B \rtimes_{\be} \bbZ$, where $B=(A/\sR_\al)_\infty$, $\be=(\dot\al)_\infty$.
\end{theorem}
\begin{proof}
It suffices to show it for $\al$ injective, i.e., $\sR_\al=(0)$. By \cite[Proposition 3.3]{Sta93} (or usual dilation theory) we obtain that $p(A_\infty \rtimes_{\al_\infty} \bbZ)p \simeq A \times_\al^1 \bbN$ which we show is $\cenv(A_{\textup{nd}} \times_\al \sT_1^+)$. To this end, fix a pair $(\pi,U)$ that integrates to a faithful representation of $A_\infty \rtimes_{\al_\infty} \bbZ$.

As in the proof of Theorem \ref{T: unit}, we can use the same gauge-invariance argument to show that the \v{S}ilov ideal intersects with the fixed point algebra $pA_\infty p$ and consequently intersects with $pAp =A$, by moving backwards with $\ad_{(U^np)^*}$, which leads to a contradiction.

To prove that it is full, let an ideal $\sI$ of $A_\infty \rtimes_{\al_\infty} \bbZ$ containing the corner. Then $\sI$ contains a copy of $A$. If $(e_i)$ is a c.a.i. of $A_\infty$ then
\begin{align*}
\pi[0,x,\al(x),\dots] = \lim_i U^* \pi(e_i) \cdot \pi[x,\al(x),\dots] \cdot \pi(e_i) U \in \sI.
\end{align*}
Inductively $A_\infty \subseteq \sI$. Therefore $\pi(x)U^n = \lim_i \pi(x) \cdot \pi(e_i)U^n \in \sI$, for all $x\in A_\infty$. Thus $\sI$ contains the generators of $A_\infty \rtimes_{\al_\infty} \bbZ$, hence it is trivial.
\end{proof}

\subsubsection{The tensor algebra $\sT_{\sM_L}^+$}

By \cite{KatsKribs06} the $\ca$-envelope of any tensor algebra is the Cuntz-Pimsner algebra. Under our assumptions for $(A,\al,L)$ we show that $\sO_{\sM_L}$ is a corner of a usual crossed product. In addition it has a special form, which we will use in the multivariable case.

\begin{theorem}\label{T:corner}
Let an Exel system $(A, \al, L)$ that satisfies $(\dagger)$ and $L(1)=1$. Then $\sO_{\sM_L}$ is a full corner of $A_\infty \rtimes_{\al_\infty} \bbZ$. In particular $p(A_\infty \rtimes_{\al_\infty} \bbZ)p$ is the closed linear span of
\[
\pi(x_0) + \sum_{n=0}^k \pi(x_n)U^np + \sum_{n=1}^k pU^{-n} \pi(y_n), \text{ for } x_0, x_n, y_n \in A.
\]
\end{theorem}
\begin{proof}
The canonical identification $\sO_{\sM_L} \simeq A \times_\al^1 \bbN$ gives that $\sT_{\sM_L}^+$ is completely isometrically isomorphic to $A_{\textup{nd}} \times_\al \sT_1^+$. Hence the first part of the theorem follows by Theorem \ref{T: not unit} and \cite{KatsKribs06}. The second part follows readily by Theorem \ref{T:Exel}, since $\sM_L$ is a Hilbert bimodule. Indeed, in this case $\sO_{\sM_L} = \overline{\sT_{\sM_L}^+ + (\sT_{\sM_L}^+)^*}$.
\end{proof}

\begin{remark}\label{R:multi lin}
When $(A,\al,L)$ satisfies $(\dagger)$ and $L(1)=1$, then $\al^n\circ L^n(x)= \al^n(1) x \al^n(1)$, for all $x\in A$, hence $pA_\infty p =A$. Therefore the monomials $p\pi(x)U^np$ and $pU^{-n} \pi(x) p$, for $x\in A_\infty$, that span $p(A_\infty \rtimes_{\al_\infty} \bbZ)p$ are written as $\pi(x')\pi(\al^n(1))U^n$ and $U^{-n} \pi(\al^n(1)) \pi(x')$, for $x' \in A$. This gives an ad-hoc proof of Theorem \ref{T:corner} and we record it for further use in Theorem \ref{T:tens}.
\end{remark}

\section{Multivariable Case}\label{S:multi}

We wish to explore universal non-selfadjoint operator algebras that arise from dynamical systems, subject to covariant relations induced by a Cuntz family.
The representation theory of such objects was exploited by Cuntz \cite{Cun77}, Murphy \cite{Mur02}, Stacey \cite{Sta93}, and Laca \cite{Lac93} for \emph{the twisted crossed product} and \emph{the crossed product of multiplicity $n$}.

\subsection{The Twisted Crossed Product}\label{Ss:twisted}

Let $\al\in \Aut(A)$; then the usual crossed product $A \rtimes_\al \bbZ$ is defined, and we simply write $(A\rtimes_\al \bbZ)\otimes \sO_n$ for the unique tensor product $\ca$-algebra.
If $(\pi_u,U_u)$ is the universal unitary covariant pair for $(A,\al)$, then the \emph{twisted crossed product of $A$ by $\al$}, denoted by $A\rtimes_\al\sO_n$, is the $\ca$-subalgebra of $(A\rtimes_\al \bbZ)\otimes \sO_n$ generated by $\pi_u(A)$ and the Cuntz family $\{U_u\otimes S_{i}\}$.
By convention $\ca(\pi_u,U_u)$ commutes with $\sO_n$. As a consequence, the embedding of $\sO_n$ in $A \rtimes_\al \sO_n$ is not multiplicity free.

Let $(\pi,U)$ acting on $H_1$ such that $\ca(\pi,U) \simeq A\rtimes_\al \bbZ$ and let $\sO_n$ acting on $H_2$. Then $A\rtimes_\al \sO_n$ is $*$-isomorphic to the $\ca$-algebra in $\sB(H_1\otimes H_2)$ generated by $\pi(A)\otimes I_{H_1}$ and $\{U\otimes S_i\}$.
Therefore the twisted crossed product does not depend on the choice of the covariant pair $(\pi,U)$, but on the dynamical system $(A,\al)$, in accordance with the original definition \cite{Cun77}. We remark that the notation $A \rtimes_\al \sO_n$ is due to Stacey \cite{Sta93} because of this fact.

\subsection{Multiplicity $n$ Crossed Products}\label{S:multi n sta}

The \emph{multiplicity $n$ crossed product} is the universal $\ca$-algebra generated by $\pi(A)T_\nu T_\mu^*$ such that
\begin{enumerate}
\item $\pi$ is a non-degenerate representation of $A$,
\item $\{T_i\}$ is a Toeplitz-Cuntz family, and
\item $\pi(\al(x))= \sum_{i=1}^n T_i \pi(x) T_i^*$, for all $x\in A$,
\end{enumerate}
and it is denoted by $A\times_\al^n \bbN$ \cite{Sta93}. It is immediate, due to non-degeneracy, that $A\times_\al^n \bbN$ is generated by $\pi(A)$ and $\{T_i\}$. Analogous to the one variable case we remove non-degeneracy by ``weighting'' the isometries $T_i$.

\begin{proposition}\label{P:multi sta}
The multiplicity $n$ crossed product $A \times_\al^n \bbN$ is $*$- isomorphic to the $\ca$-subalgebra $\ca(\pi_u(A), \{T_{u,i}\pi_u(1)\})$ of the universal $\ca$-algebra $\ca(\pi_u(A), \{T_{u,i}\})$ relative to $(\pi,\{T_i\})$ such that
\begin{enumerate}
\item $\pi$ is a representation of $A$,
\item $\{T_i\}$ is a Toeplitz-Cuntz family, and
\item $\pi(\al(x))= \sum_{i=1}^n T_i \pi(x) T_i^*$, for all $x\in A$.
\end{enumerate}
In particular, $A \times_\al^n \bbN$ coincides with $\ca(\pi_u(A), \{T_{u,i}\})$ if and only if $\al(1)=1$.
\end{proposition}
\begin{proof}
If $\pi_u(\al(x))= \sum_{i=1}^n T_{u,i} \pi_u(x) T_{u,i}^*$, then $\pi_u(\al(x))T_{u,i} = T_{u,i} \pi_u(x)$, for all $x\in A$ and $i=1,\dots,n$. Hence $\pi_u(1) T_{u,i} \pi(1) = T_{u,i} \pi_u(1)$ for all $i=1,\dots,n$. The proof is completed by arguments similar to those of Propositions \ref{P:sta-exe} and \ref{P:criterion}.
\end{proof}

\begin{remark}\label{R:cex 1}
The set of representations of $A_\infty \rtimes_{\al_\infty} \sO_n$ are sufficient to obtain an injective embedding of $A$ in $A\times_\al^n \bbN$.
However, they may not be enough to obtain the norm of $A\times_\al^n \bbN$. In other words the canonical $*$-homomorphism $A\times_\al^n \bbN \rightarrow A_\infty \rtimes_{\al_\infty} \sO_n$ may not be injective, as we show in the counterexample in Subsection \ref{S:counter}. Therefore there is no analogue of \cite[Proposition 3.3]{Sta93} for $n>1$.
\end{remark}

\subsection{The Semicrossed Product $A \times_\al \sT_n^+$}\label{S:multi spr}

In what follows we construct an operator algebra $A \times_\al \sT_n^+$ starting with a Banach algebra $\ell^1(A,\al, \bbF_n^+)$. This universal operator algebra should not be mistaken with the objects examined in \cite{DavKatsMem, DunPet10, Ful11}.

Let the semigroup $\bbF_n^+$ and the linear tensor product $A\otimes c_{00}(\bbF_n^+)= \spn\{ x\otimes \de_\mu: x\in A, \mu\in \bbF_n^+\}$. The multiplication rule that turns $A \otimes c_{00}(\bbF_n^+)$ into an algebra is
\begin{align*}
(x\otimes \de_i) \cdot (y\otimes \de_j)= (x\al(y))\otimes \de_{ij},
\end{align*}
for $i, j \in \{1, \dots, n\}$.
Define the $\ell^1$-norm
\begin{align*}
\left| \sum_\mu x_\mu \otimes \de_\mu\right|_1 = \sum_\mu
\nor{x_\mu}_A,
\end{align*}
where $\mu$ is taken over a finite subset of finite paths in $\bbF_n^+$, and let $\ell^1(A,\al, \bbF_n^+)$ be the Banach algebra that comes from the $\ell^1$-completion of $A\otimes c_{00}(\bbF_n^+)$. Note that
$
\ell^1(A,\al, \bbF_n^+)= \overline{\spn\{ x\otimes \de_\mu: x\in A, \mu\in
\bbF_n^+\}}^{|\cdot|_1}.
$
When $\al(1)=1$, then $1\otimes \de_\emptyset$ is the unit for $\ell^1(A,\al, \bbF_n^+)$, and $\ell^1(A,\al, \bbF_n^+)= \overline{\alg}^{|\cdot|_1}\{x\otimes \de_\emptyset, 1 \otimes \de_\mu: x\in A, \mu \in \bbF_n^+\}$.

Consider the class
\begin{align*}
\sF= \big\{ (\pi, \{T_i\}_{i=1}^n) \mid \pi\colon A \rightarrow \sB(H),\, T_i \in
\sT_n^+,\, \pi(\al(x))=\sum_{i=1}^nT_i\pi(x)T_i^*\big\},
\end{align*}
If $(\pi,\{T_i\}) \in \sF$ then $T_\mu \pi(x)= \pi(\al^{|\mu|}(x)) T_\mu$, for all $x\in A$, and $\mu \in \bbF_n^+$.
Moreover $(\pi, \{T_i\}_{i=1}^n)$ defines a representation $\rho\colon \ell^1(A,\al, \bbF_n^+) \rightarrow \sB(H)$, such that $\rho(x\otimes \de_\mu)= \pi(x)T_\mu$. Indeed, $\rho$ is $|\cdot|_1$-contractive on $A\otimes c_{00}(\bbF_n^+)$ (and extends to $\ell^1(A,\al, \bbF_n^+)$) since
\begin{align*}
\nor{\rho(\sum_\mu x_\mu \otimes \de_\mu)}
\leq \sum_\mu \nor{\rho(x_\mu) T_\mu}
\leq \sum_\mu \nor{x_\mu}_A
= \left| \sum_\mu x_\mu \otimes \de_\mu\right|_1.
\end{align*}
In this case we say that $\rho$ \emph{integrates} $(\pi,\{T_i\})$.

\begin{definition}\label{D:Twisted scp}
We denote by $A \times_\al \sT_n^+$ the universal operator algebra relative to the class $\sF$ that is generated by $\pi(x)$ and $T_\mu$, for $x\in A$ and $\mu \in \bbF_n^+$.
\end{definition}

\begin{remark}
The semicrossed product $A \times_\al \sT_n^+$ constructed here differs in general from the tensor algebra relative to $n$ $*$-endomorphisms $\al_i$ of $A$, examined in \cite{MS, DavKatsMem, KK}. Nevertheless, when $\al_i$ coincide with an automorphism $\al$, then they are completely isometrically isomorphic.
\end{remark}

The class $\sF$ of representations is non-empty. Moreover, the existence of such representations is equivalent to $A_\infty \neq (0)$ \cite[Proposition 2]{Sta93}. We present such an example, which we prove that has Fourier co-efficients.

\begin{example}\label{E:rpn}
Let $(\pi,U)$ be the left regular representation of the dynamical system $(A_\infty,\al_\infty)$ acting on $H$. Let $\sK= H \otimes K$, where $K$ is a Hilbert space on which $\sO_n$ acts (faithfully). If $\{S_i\}$ is a Cuntz family on $K$, let $T_i=U \otimes S_i$ and $\widehat{\pi}=\pi\otimes I_K$. Then $\{T_i\}$ is also a Cuntz family. Moreover if $q\colon A \rightarrow A_\infty$ is the canonical $*$-homomorphism, then
\begin{align*}
\sum_{i=1}^n T_i \widehat{\pi}\circ q(c) T_i^*
&= \sum_{i=1}^n (U\pi[c,\al(c),\dots]U^*) \otimes S_iS_i^*\\
&= \pi\circ \al_\infty [c,\al(c),\dots] \otimes I
 = \widehat{\pi}\circ q(\al(c)),
\end{align*}
for $i=1,\dots,n$. Hence $(\widehat{\pi}\circ q, \{T_i\})$ defines a pair in $\sF$. Note that $I_H \otimes S_i \in \widehat{\pi}(A_\infty)'$; therefore, for a finite sum $\sum_{\mu \in F} \widehat{\pi}(x_\mu)T_\mu$, we define
\begin{align*}
P_\nu(\sum_{\mu \in F} \widehat{\pi}\circ q(x_\mu)T_\mu)
& := (I\otimes S_\nu^*) \cdot \left(\sum_{\mu \in F} \widehat{\pi}\circ q(x_\mu)T_\mu\right) \cdot (U^{-|\nu|}\otimes I_K)\\
&  =\sum_{\mu \in F} \widehat{\pi}\circ q(x_\mu) \cdot (I\otimes S_\nu^*)T_\mu(U^{-|\nu|}\otimes I_K)\\
&  = \widehat{\pi}\circ q(x_\nu)\cdot (U^{|\nu|} \otimes I_K)(U^{-|\nu|}\otimes I_K)
   = \widehat{\pi}\circ q(x_\nu),
\end{align*}
since $T_{\mu} = U^{|\mu|}\otimes S_{\mu}$.
Thus $P_\nu(\sum_{\mu \in F} \widehat{\pi}\circ q(x_\mu)T_\mu)=0$, for every $\nu \in F$, if and only if $\widehat\pi\circ q(x_\nu)=0$ for every $\nu\in F$, if and only if $x_\nu \in \sR_\al$ for every $\nu\in F$, if and only if $\sum_{\mu \in F} \widehat{\pi}\circ q(x_\mu)T_\mu=0$.
\end{example}

\begin{proposition}\label{P:kernel}
The kernel $\sN$ of the seminorm $\om_1$ relative to $\sF$ coincides with $\ell^1(\sR_\al, \al, \bbF_n^+)$.
\end{proposition}
\begin{proof}
First observe that
\begin{align*}
\ell^1(\sR_\al, \al, \bbF_n^+)= \overline{\spn}^{|\cdot|_1}\{ x\otimes \de_\mu: x\in
\sR_\al, \mu\in \bbF_n^+\}
\end{align*}
can be embedded as an ideal in $\ell^1(A, \al, \bbF_n^+)$ and that it is contained in $\sN$.

For the converse, assume that $X=\sum_{\mu \in F} x_\mu \otimes \de_\mu$ is in $\sN$, where $F$ is finite, and the sum is written in reduced form with respect to the words $\mu$. Let $\rho$ be the representation of Example \ref{E:rpn}. By definition of $\sN$ we have that $\rho(X)=\sum_{\mu \in F} \widehat{\pi}\circ q(x_\mu)T_\mu =0$; then $x_\nu \in \sR_\al$, by Example \ref{E:rpn}. Hence $X \in \ell^1(\sR_\al,\al,\bbF_n^+)$. To end the proof, note that an $X\in \sN$ is the $\ell^1$-sum of such elements and that $\ell^1(\sR_\al,\al,\bbF_n^+)$  is $\ell^1$-closed.
\end{proof}

Analogous to the one variable case \cite[Proposition 3.4]{Kak11}, we can always assume that $(A,\al)$ is injective.

\begin{proposition}\label{P:only inj n}
The operator algebras $A \times_\al \sT_n^+$ and $(A/\sR_\al)
\times_{\dot\al} \sT_n^+$ are completely isometrically isomorphic.
\end{proposition}

We mention that the unit $1\in A$ may not be the unit of $A \times_\al \sT_n^+$, since $A$ is not represented non-degenerately. Below we give the exact non-selfadjoint analogue of $A \times_\al^n \bbN$.

\begin{definition}
Let $A_{\textup{nd}} \times_\al \sT_n^+$ be the universal operator algebra relative to the class
\begin{align*}
\sF_{\text{nd}}= \big\{ (\pi, \{T_i\}_{i=1}^n) \in \sF \mid \pi \text{ is non-degenerate}\},
\end{align*}
generated by $\pi(x)T_\mu$, for $x\in A, \mu \in \bbF_n^+$.
\end{definition}

Due to non-degeneracy $A_{\text{nd}} \times_\al \sT_n^+$ is generated separately by a copy of $A$ and a copy of $\sT_n^+$. Moreover, the $\ca$-algebra generated by the universal representation of $A_{\text{nd}} \times_\al \sT_n^+$ is $A \times_\al^n \bbN$. Analogous to the $\ca$-case for $n=1$ we obtain the following.

\begin{proposition}
The operator algebras $A_{\textup{nd}} \times_\al \sT_n^+$ and $(A/\sR_\al)_{\textup{nd}} \times_{\dot\al} \sT_n^+$ are unital completely isometrically isomorphic. Moreover, $A_{\textup{nd}} \times_\al \sT_n^+$ is unital completely isometrically isomorphic to the subalgebra generated by $\pi_u(A)$ and $\{T_{u,i}\pi_u(1)\}$ in $A \times_\al \sT_n^+$.
\end{proposition}

There is a considerable difference between the unital and non-unital dynamical systems. One such difference is highlighted below.

\begin{remark}
Fix $(\pi, \{T_i\}_{i=1}^n) \in \sF$. When $\al(1)=1$, then $\pi(1)$ is a unit for $\ca(\pi(A), T_i)$, hence $\pi$ can be chosen unital. Moreover $\{T_i\}$ is automatically a Cuntz family. Thus $(\pi, \{T_i\}) \in \sF$ if and only if $T_i \in \sO_n^+$ and $\pi(\al(x)T_i=T_i\pi(x)$.
Therefore, $A \times_\al \sT_n^+$ is generated by $A/\sR_\al$ and $\sO_n^+$.

Additionally, the algebras $A\times_\al \sT_n^+$ and $A_{\text{nd}} \times_\al \sT_n^+$ coincide. Hence the $\ca$-algebra generated by the universal representation of $A \times_\al \sT_n^+$ coincides with $A \times_{\al}^n \bbN$.
\end{remark}

\begin{proposition}\label{P:Fmaximal}
Let $(A,\al)$ be a unital dynamical system. Then every completely isometric representation $\rho$ that integrates a family $(\pi, \{S_i\})\in \sF$ is maximal.
Therefore $A \times_\al \sT_n^+$ has the unique extension property.
\end{proposition}
\begin{proof}
Without loss of generality we can assume that $\al$ is injective. Hence by the previous remarks the algebra $A \times_\al \sT_n^+$ is generated by $A$ and $\sO_n^+$. Let $\rho\colon A \times_\al \sT_n^+ \rightarrow \sB(H)$ be a representation that integrates a pair $(\pi, \{S_i\})$ in $\sF$ and let $\nu\colon A \times_\al \sT_n^+ \rightarrow \sB(K)$ be a dilation of $\rho$. Then $\nu|_A$ is a dilation of the $*$-representation $\rho|_A$, hence $\nu|_A$ is trivial. That is $H$ is $\nu(A)$-reducing. It suffices to show that $\rho|_{\sO_n^+}$ is also maximal, for then $H$ will be $\nu(\sO_n^+)$-reducing. This will imply that $H$ is $\nu(A \times_\al \sT_n^+)$-reducing, hence $\rho$ is maximal.

\vspace{.05in}

\noindent \textit{Claim.} The tensor algebra $\sT_n^+$ has the unique extension property.

\noindent \textit{Proof of the Claim.} Let $\nu\colon \sO_n^+ \rightarrow \sB(K)$ be a maximal dilation of $\rho$, which is a unital completely isometric map. Then $\nu$ extends to a unique (faithful) $*$-representation of $\sO_n$, which we will again denote by $\nu$. Hence $\{\nu(S_i)\}$ is also a Cuntz family. For $i=1,\dots,n$, let
\begin{align*}
\nu(S_i)= \left[ \begin{array}{cc} \rho(S_i) & a_i\\ 0  & b_i \end{array} \right].
\end{align*}
Note that the $(2,1)$-entry must be zero since $\rho(S_i)$ is an isometry and $\nu(S_i)$ is a contraction. Then
\begin{align*}
\nu(S_i)\nu(S_i)^*= \left[ \begin{array}{cc}
\rho(S_i)\rho(S_i)^* + a_ia_i & \ast\\
\ast & \ast\\
\end{array} \right].
\end{align*}
and $\sum_{i=1}^n \nu(S_i)\nu(S_i)^* = I_K$, since $\{\nu(S_i)\}$ is a Cuntz family. Thus, by equating the $(1,1)$-entries
\begin{align*}
\sum_{i=1}^n \rho(S_i)\rho(S_i)^* + a_ia_i^* = I_H.
\end{align*}
But $\{\rho(S_i)\}$ is in turn a Cuntz family on $H$, thus $0\leq a_ia_i^* \leq \sum_{i=1}^n a_ia_i^* = 0$.
Hence $a_i=0$ for every $i=1,\dots,n$, and $\nu$ is a trivial dilation of $\rho$.
\end{proof}

The following Corollary is immediate. Note that the operator algebra we have constructed may not be separable (cf. Subsection \ref{Ss:opalg}).

\begin{corollary}\label{C:maximal}
Let $(A,\al)$ be a unital dynamical system. Then the Choquet boundary exists for the operator algebra $A \times_\al \sT_n^+$.
\end{corollary}

\begin{theorem}\label{T: not unit multi}
Let $(A,\al)$ be a unital dynamical system. Then $\cenv(A\times_\al \sT_n^+) \simeq A \times_\al^n \bbN$.
\end{theorem}
\begin{proof}
In view of Corollary \ref{C:maximal} and Lemma \ref{L:fful_c*-env}, we have that the identity representation of $A\times_\al^n \bbN$ is faithful and $A\times_\al \sT_n^+$ has the unique extension property, thus $A\times_\al^n \bbN \simeq \cenv(A \times_\al \sT_n^+)$.
\end{proof}

\begin{theorem}\label{T:env *-auto}
Let $(A,\al)$ be an automorphic dynamical system. Then the $\ca$-envelope of $A \times_\al \sT_n^+$ is the twisted crossed product $A \rtimes_\al \sO_n$.
\end{theorem}
\begin{proof}
By assumption $\al(1)=1$, thus the $\ca$-algebra generated by the universal representation of $A \times_\al \sT_n^+$ is $A \times_\al^n \bbN$. Then \cite[Proposition 3.4]{Sta93} completes the proof.
\end{proof}

\begin{remark}\label{R:cex 2}
Following the results for $n=1$ of \cite{Kak11}, and the previous Theorem \ref{T:env *-auto}, one might speculate that the $\ca$-envelope of $A \times_\al \sT_n^+$ is (a corner of) the twisted crossed product $A_\infty \rtimes_{\al_\infty} \sO_n$ (at least when $\al$ is injective). The counterexample constructed in Subsection \ref{S:counter} shows that this is not true in general, even when $\al$ is unital.
\end{remark}

\begin{remark}
Alternatively one could define the universal object associated to $\ell^1(A,\al,\bbF_n^+)$ relative to the class
\begin{align*}
\big\{ (\pi, \{T_i\}_{i=1}^n) \mid \pi\colon A \rightarrow \sB(H),\, T_i \in
\sT_n^+,\, \pi(\al(x))T_i=T_i\pi(x)\big\}.
\end{align*}
It is easy to check that the previous analysis applies to this case also.
\end{remark}

\subsection{The Semicrossed Product $A \times_\al^t \sT_n^+$}\label{S:multi twi spr}

We present the non-selfadjoint analogue of the twisted crossed product. In the next Subsection we give Exel systems $(A,\al,L)$ for which the class $\sF^t$, that follows, is non-empty.

\begin{definition}
Let $\sF^t$ be the subclass of
\begin{align*}
\left\{ (\pi,S,\{T_i\}) \mid \ S \text{ isometry},\, S\pi(x)=\pi(\al(x))S,\, T_i\in \sT_n^+ \right\}.
\end{align*}
such that $(\pi, S, \{S_i\})$ integrates to a representation $\rho$ of $\ell^1(A,\al,\bbF_n^+)$ with
\begin{align*}
\rho(x \otimes \de_i)= \pi(x)S \otimes S_i, \foral x\in A, i=1,\dots,n.
\end{align*}
We denote by $A \times_\al^t \sT_n^+$ the universal operator algebra relative to the class $\sF^t$, generated by $\pi(x)\otimes I$ and $S^{|\mu|}\otimes T_{\mu}$, for $x\in A$ and $\mu\in \bbF_n^+$.
\end{definition}

\begin{remark}\label{R:cex 3}
We mention that the canonical map $A \times_\al \sT_n^+ \rightarrow A \times_\al^t \sT_n^+$ may not be a complete isometry, as we show within the counterexample of Subsection \ref{S:counter}.
\end{remark}

\begin{theorem}\label{T:twi spr}
The $\ca$-envelope of the operator algebra $A \times_\al^t \sT_n^+$ is the $\ca$-subalgebra $\ca(B,\{U\otimes S_i\})$ of the twisted tensor product $\sM(B) \rtimes_{\be} \sO_n$, where $B= (A/\sR_\al)_\infty$ and $\be =(\dot{\al})_\infty$.
\end{theorem}
\begin{proof}
It is easy to deduce that $A \times_\al^t \sT_n^+$ is unital completely isometrically isomorphic to $A/\sR_\al \rtimes_{\dot\al} \sT_n^+$. Therefore we can always assume that the endomorphism $\al$ is injective.

\vspace{.05in}

\noindent \textit{Claim 1.} The $\ca$-algebra $\ca(A_\infty,\{U\otimes S_i\})$ is a $\ca$-cover for $A \times_\al^t \sT_n^+$.

\noindent \textit{Proof of Claim 1.} Let $(\pi,S, \{T_i\})\in \sF^t$ acting on $H_1\otimes H_2$; then we can extend the pair $(\pi,S)$ of $(A,\al)$ to a covariant unitary pair $(\Pi,U)$ (acting on an $H_3$) of the dynamical system $(A_\infty,\al_\infty)$ by \cite[Proposition 2.3]{Sta93} and dilate $\{T_i\}$ to a Cuntz family $\{S_i\}$ separately. Then $(\Pi, U, \{S_i\})$ defines a representation of $A_\infty \rtimes_{\al_\infty} \sO_n$. Conversely, if $(\Pi,U)$ is a covariant unitary pair of $(A_\infty, \al_\infty)$ and $\{S_i\}$ a Cuntz family, then $(\Pi|_A, U, \{S_i\})$ defines a representation of $A \times_\al^t \sT_n^+$. Hence, there is a canonical completely isometric embedding $A \times_\al^t \sT_n^+ \hookrightarrow \sM(A_\infty) \rtimes_{\al_\infty} \sO_n$.

To end the proof of Claim $1$, note that $\pi(A)\otimes I$ and $U\otimes S_i$ generate the twisted product. Indeed, it suffices to show that they generate $\pi(A_\infty)\otimes I$; this is proved by noting that
\begin{align*}
  \pi[0,x,\al(x), \dots]\otimes I = (U^*\pi(x)U)\otimes I= (U^*\otimes S_i^*) (\pi(x)\otimes I) (U\otimes S_i),
\end{align*}
for all $x\in A$, and then using induction, since $\pi(A_\infty)= \overline{\cup_n U^{-n}\pi(A) U^n}$.

\vspace{.05in}

\noindent \textit{Claim 2.} The representations of $\sM(A_\infty) \rtimes_{\al_\infty} \sO_n$ are boundary for $A \times_\al^t \sT_n^+$.

\noindent \textit{Proof of Claim 2.} Let $\rho \equiv (\pi, \{U\otimes S_i\})$ be a representation of $\sM(A_\infty) \rtimes_{\al_\infty} \sO_n$ acting on a Hilbert space $H$, and let $\nu$ be a dilation of the restriction of $\rho$ to $A \times_\al^t \sT_n^+$ acting on a Hilbert space $K$. Since $\rho|_A$ is a $*$-representation we get that $H$ is $\nu(A)$-reducing. Also, $\{U \otimes S_i\}$ is a Cuntz family, hence $\ca(U\otimes S_i)$ is isomorphic to $\sO_n$. Therefore, $\rho$ defines a unital completely isometric representation of $\sO_n^+$ thus, as in Proposition \ref{P:Fmaximal}, it is maximal. Hence $H$ is reducing also for $\nu(\sO_n^+)$. Therefore $\nu$ is a trivial dilation, thus $\rho$ is maximal.
\end{proof}

\begin{remark}
In view of the theory developed in \cite{Kak11} one could define the universal non-selfadjoint object relative to the class
\begin{align*}
\sF^t_{\text{un}}=\left\{ (\pi,U,\{T_i\}) \in \sF^t \mid U \text{ unitary} \right\}.
\end{align*}
The proof of Theorem \ref{T:twi spr} shows then that it is completely isometrically isomorphic to $A \times_\al^t \sT_n^+$, thus they share the same $\ca$-envelope. When $\al$ is injective, $U$ can be considered alternatively a contraction \cite{Kak11}.

If in addition $\al$ is assumed to be an automorphism they are unital completely isometrically isomorphic to $A \times_\al \sT_n^+$. Indeed, the latter is true because $\cenv(A \times_\al \sT_n^+)= A\rtimes_\al \sO_n = \cenv(A \times_\al^t \sT_n^+)$, via a $\ca$-isomorphism that fixes the non-selfadjoint parts \cite[proof of Proposition 3.4]{Sta93}.
\end{remark}

In the definition of $A \times_\al^t \sT_n^+$ we imposed that it is generated by $\pi_u(A)$ and $\{S_u\otimes S_{u,i}\}$. This is not the algebra generated by $\pi_u(x)S_u\otimes S_{u,i}$, unless $\al$ is unital. Following the generalized $\ca$-crossed products' theory, another way to impose this is the following.

\begin{definition}
Let $A_{\text{nd}} \times_\al^t \sT_n^+$ be the universal operator algebra of $\ell^1(A,\al,\bbF_n^+)$ relative to the class
\begin{align*}
\sF_{\text{nd}}^t=\left\{ (\pi,S,\{T_i\}) \in \sF^t \mid \pi \text{ is non-degenerate} \right\}.
\end{align*}
As a consequence $A_{\text{nd}} \times_\al^t \sT_n^+$ is generated by $\pi(x)\otimes I$ and $S^{|\mu|}\otimes T_{\mu}$, for $x \in A$ and $\mu \in \bbF_n^+$.
\end{definition}

As in Proposition \ref{P:sta-exe} $A_{\text{nd}} \times_\al^t \sT_n^+$ is unital completely isometrically isomorphic to the subalgebra of $A \times_\al^t \sT_n^+$ generated by $\pi_u(A)\otimes I$ and the family $\{S_u\pi_u(1) \otimes T_{u,i}\}$. Indeed, this is true for the $\ca$-algebras generated by the universal representations, whose restriction implies the above identification. Therefore, if $(\pi,U,\{S_i\})$ is a faithful representation of $\sM(B) \rtimes_\be \sO_n$ as in Theorem \ref{T:twi spr}, then $(\pi, U\pi(1), \{S_i\})$ integrates to a maximal completely isometric representation of $A_{\text{nd}} \times_\al^t \sT_n^+$.

\begin{theorem}\label{T:twi spr 2}
The $\ca$-envelope of $A_{\textup{nd}} \times_\al^t \sT_n^+$ is a full corner of $B \rtimes_\be \sO_n$, where $B=(A/\sR_\al)_\infty$, $\be=(\dot\al)_\infty$.
\end{theorem}
\begin{proof}
For the first part, in view of the previous remarks (and the -by now- familiar trick), it suffices to show that, when $\al$ is injective, then $\ca(\pi(A)\otimes I, \{U\pi(1) \otimes S_i\})$ coincides with $p(A_\infty \rtimes_{\al_\infty} \sO_n)p$, where $p=\pi(1)\otimes I$.

Since $U\pi(1) \otimes S_i = p(U\otimes S_i)p$ we obtain that $\ca(\pi(A)\otimes I, \{U\pi(1) \otimes S_i\})$ is a $\ca$-subalgebra of $p(A_\infty \rtimes_{\al_\infty} \sO_n)p$.

For the converse, first recall that $A_\infty \rtimes_{\al_\infty} \sO_n$ is generated by $\pi(A_\infty) \otimes I$ and $\pi(x)U \otimes S_i$ for $x\in A_\infty$ and $i=1, \dots, n$. Hence the generators of $p(A_\infty \rtimes_{\al_\infty} \sO_n)p$ are $p(\pi(A_\infty)\otimes I)p$ and
\[
p(\pi(x)U\otimes S_i)p = (\pi(1) \pi(x)U \pi(1) )\otimes S_i = p(\pi(x)\otimes I) p \cdot (U\pi(1) \otimes S_i).
\]
Due to that equation, it suffices to show that $p(\pi(A_\infty)\otimes I)p$ is in $\ca(\pi_u(A)\otimes I, \{U\pi(1) \otimes S_i\})$, which follows as in the proof of Claim $1$ of Theorem \ref{T:twi spr}.

To prove that $p(A_\infty \rtimes_{\al_\infty} \sO_n)p$ is a full corner, let an ideal $\sI$ of $A_\infty \rtimes_{\al_\infty} \sO_n$ containing $p(A_\infty \rtimes_{\al_\infty} \sO_n)p$. Then $\pi(A)\otimes I \subseteq \sI$ and for an approximate unit $(e_i)$ of $A_\infty$ and $x\in A$ we obtain
\begin{align*}
\pi[0,x,\al(x), \dots]\otimes I = \lim_i (U^*\pi(e_i)\otimes S_i^*)\cdot (\pi(x)\otimes I) \cdot (\pi(e_i)U\otimes S_i) \in \sI.
\end{align*}
Inductively, we get that $\pi(A_\infty)\otimes I \subseteq \sI$. Therefore,
\begin{align*}
\pi(x)U^{|\mu|}\otimes S_\mu = \lim_i \pi(x)\otimes I \cdot \pi(e_i)U^{|\mu|} \otimes S_\mu,
\end{align*}
for every $x\in A_\infty$ and $\mu\in \bbF_n^+$. Thus the non-zero ideal $\sI$ contains the generators of $A_\infty \rtimes_{\al_\infty} \sO_n$, hence it is a trivial ideal.
\end{proof}

\begin{remark}
The previous theorem asserts that a generalization of \cite[Proposition 3.3]{Sta93} for $n>1$ holds for the $\ca$-envelope of $A_{\textup{nd}} \times_\al^t \sT_n^+$. This agrees also with Theorem \ref{T: not unit} for $n=1$.
\end{remark}

As in the one variable case, it is immediate that $A_\text{nd} \times^t_\al \sT_n^+ = A \times_\al^t \sT_n^+$ if and only if $\al(1_A)=1_A$. In the following subsection we examine a class of non-unital systems, so that $A \times_\al^t \sT_n^+$ and $A_{\text{nd}} \times_\al^t \sT_n^+$ differ.

\subsection{Tensor Products with Exel Systems}\label{S:tensor}

Let an Exel system $(A,\al,L)$ and $X_n$ be the $\ca$-correspondence of the Hawaiian earring graph on $n$ edges, over $\bbC$. Let $\sX=\sM_L\otimes X_n$ be the exterior tensor product, which becomes a $\ca$-correspondence over $A\otimes \bbC \simeq A$ in the obvious way, with $\phi_\sX= \phi_{\sM_L}\otimes \id$.

If $(\rho,s)$ is a representation of $X_n$ and $(\pi,t)$ is a representation of $\sM_L$ then $(\pi\otimes \rho, t\otimes s)$ defines a representation of $\sM_L\otimes X_n$, which is injective when $\pi$ is injective. Moreover, when identifying $\sK(\sX)$ with the spatial tensor product $\sK(\sM_L)\otimes \sK(X_n)$ by the $*$-isomorphism $j$, with
\[
j\left(\Theta_{y+\sN_1, x+ \sN_1}\otimes \Theta_{\de_i,\de_j} \right)= \Theta_{(y+\sN_1)\otimes \de_i, (x+ \sN_1)\otimes \de_j},
\]
then $\psi_{t\otimes s}\circ j= \psi_t\otimes \psi_s$. Thus, if $(\pi,S)$ is a representation of $\sM_L$ and $\{T_i\}$ is a Toeplitz-Cuntz family then $(\hat{\pi},\{\hat{T_i}\})$ defines a representation of $\sX$, where $\hat{\pi}=\pi\otimes I$ and $\hat{T_i}=S \otimes T_i$. Hence,
\begin{align*}
\hat{T_i} \hat{\pi}(x)
&= (S \otimes T_i)(\pi(x)\otimes I)
 = (S\pi(x))\otimes T_i\\
&= (\pi(\al(x))S) \otimes T_i
 = (\pi(\al(x))\otimes I) (S\otimes T_i) = \hat{\pi}(\al(x)) \hat{T_i}.
\end{align*}

Conversely, if $(\hat{\pi},t)$ is a representation of $\sX$, let $\hat{T_i}= t((1+\sN_1)\otimes \de_i)$. Then $\hat{T_i}\hat{\pi}(x)= \hat{\pi}(\al(x)) \hat{T_i}$, for all $i=1,\dots,n$ and $x\in A$.

\begin{lemma}\label{L:Katsura ideals}
With the above notations $J_\sX = J_{\sM_L}$, where we identify $A \otimes \bbC$ with $A$.
\end{lemma}
\begin{proof}
It is immediate that $\ker\phi_\sX= \ker\phi_{\sM_L}$. To show that $\phi^{-1}(\sK(\sX))=\phi^{-1}_{\sM_L}(\sK(\sM_L))$, first let $x\in A$, such that $\phi_{\sM_L}(x)=k\in \sK(\sM_L)$. Then
\begin{align*}
\phi_\sX(x)
& = \phi_{\sM_L}(x) \otimes \id
  = k \otimes \left(\sum_{i=1}^n \Theta^{X_n}_{\de_i,\de_i}\right)
 \, \in \, \sK(\sM_L) \otimes \sK(X_n) = \sK(\sX),
\end{align*}
for the spatial tensor product \cite[Chapter 4]{Lan95}. For $x\in A$ such that $\phi_\sX(x)= \phi_{\sM_L}(x) \otimes \id \in \sK(\sM_L)\otimes \sK(X_n)$, let $(e_i)$ be a c.a.i. of $\sK(\sM_L)$. Then $(e_i \otimes \id)$ is a c.a.i. for $\sK(\sM_L)\otimes \sK(X_n)$. Since,
\begin{align*}
\lim_i \nor{e_i\phi_{\sM_L}(x)-\phi_{\sM_L}(x)} &= \lim_i \nor{(e_i\phi_{\sM_L}(x)-\phi_{\sM_L}(x))\otimes \id}\\
&=\lim_i \nor{e_i\phi_{\sM_L}(x)\otimes \id-\phi_{\sM_L}(x)\otimes \id}\\
&=\lim_i \nor{(e_i\otimes \id)\phi_{\sX}(x)- \phi_\sX(x)}=0,
\end{align*}
we obtain that $\phi_{\sM_L}(x)\in \sK(\sX)$, and the proof is complete.
\end{proof}

Consequently, if $(\pi,t)$ is a covariant representation of $\sM_L$ and $(\rho,s)$ is a covariant representation for $X_n$, i.e., $\{s(\de_i)\}$ is a Cuntz-Krieger family, then $(\pi\otimes \rho, t\otimes s)$ is a covariant representation of $\sX$. Indeed, let $x\in J_\sX$; then $x\in J_{\sM_L}$ by Lemma \ref{L:Katsura ideals}. Thus
\begin{align*}
(\pi\otimes \rho) (x) = \pi(x)\otimes I
& = \psi_t\left(\phi_{\sM_L}(x)\right) \otimes \psi_s\left(\sum_{i=1}^n \Theta^{X_n}_{\de_i,\de_i}\right)\\
& = \psi_{t\otimes s}\circ j \left(\phi_{\sM_L}(x) \otimes \id\right) = \psi_{t\otimes s} \left(\phi_\sX(x)\right).
\end{align*}

From now on the Exel system $(A,\al,L)$ will be assumed that satisfies the conditions of Subsection \ref{S:Ex}, i.e.,
\[
 L(1)=1 \text{ and } \al\circ L (x) = \al(1)x\al(1), \foral x\in A.
\]
Let $(\pi_u, T_u)$ be the universal representation of $\fA(A,\al)$. By Proposition \ref{P:sta-exe} and Theorem \ref{T:main 2},
\[
\ca(\pi_u(A), T_u\pi_u(1)) \simeq A \times_\al ^1 \bbN \simeq A \rtimes_{\al,L}\bbN \simeq \sO_{\sM_L}
\]
via a canonical mapping, and the pair $(\pi_u, T_u\pi_u(1))$ defines a faithful covariant representation of $\sM_L$ that admits a gauge action $\{\be_z\}$. Thus, if $\{S_i\}$ is a Cuntz family, then $(\pi_u\otimes \id, \{(T_u\pi_u(1)) \otimes S_i\})$ defines a covariant representation of $\sX$ that inherits the gauge action $\{\be_z \otimes I\}$ and is faithful on $A$. Therefore by the gauge invariance theorem \cite{Kats04}, the Cuntz-Pimsner algebra $\sO_\sX$ is $*$-isomorphic to the $\ca$-algebra generated by $(\pi_u \otimes \id, \{(T_u\pi_u(1))\otimes S_i\})$. Note that it is the $\ca$-subalgebra $\ca(A_{\text{nd}} \times_\al^t \sT_n^+)$ of $\ca(A \times_\al^t \sT_n^+)$.

The next result follows readily from the previous analysis and the results we established in Subsection \ref{S:multi twi spr}.

\begin{theorem}\label{T:tens}
Let an Exel system that satisfies $(\dagger)$ and $L(1)=1$.
\begin{enumerate}
\item $\sT_{\sX} \simeq A_{\textup{nd}} \times_\al^t \sT_n^+$;
\item the $\ca$-envelope $\sO_{\sX}$ of $\sT_{\sX}^+$ is a full corner of $A_\infty \rtimes_{\al_\infty} \sO_n$;
\item If $(\pi,U, \{S_i\})$ integrates to an injective representation of $A_\infty \rtimes_{\al_\infty} \sO_n$ and $p=\pi(1)\otimes I$, then $p(A_\infty \rtimes_{\al_\infty} \sO_n)p$ is the closed linear span of the monomials $\big(\pi(x) U^{|\mu|-|\nu|} \pi(1) \big)\otimes S_\mu S_\nu^* $, when $|\mu|\geq |\nu|$, and $\big(\pi(1)U^{|\mu|-|\nu|}\pi(x) \big)\otimes S_\mu S_\nu^*$, when $|\mu|\leq |\nu|$, for $x\in A$.
\end{enumerate}
\end{theorem}
\begin{proof}
It suffices to comment on item $(3)$. Its proof follows in the same way as the second part of Theorem \ref{T:corner}, as explained in Remark \ref{R:multi lin}. (Recall that $p(A_\infty)p=A$ in this case.)
\end{proof}

\begin{remark}\label{R:cex 4}
We mention that in the representation theory of $A \times_\al^t \sT_n^+$, $A_{\text{nd}} \times_\al^t \sT_n^+$ and $\sT_\sX$, the commutant of $\sO_n$ is non-trivial. This provided an isometry $S$ that commutes with the isometries $S_i$, hence we could extend each $(\pi,S)$ to a unitary pair $(\Pi,U)$ without losing control on the isometries $S_i$. The counterexample that follows shows why this step is crucial in order to prove a connection with the twisted tensor product of $(A_\infty,\al_\infty)$.
\end{remark}

\subsection{Counterexample}\label{S:counter}

Let $A=\sO_n$ acting faithfully on some $H$ and let us denote by $S_i$ its generators. Define the $*$-endomorphism $\al$ of $\sO_n$ by
\begin{align*}
\al(x)= \sum_{i=1}^n S_i x S_i^*, \foral x\in A.
\end{align*}
This dynamical system is unital and injective. Therefore $A\times_\al^n \bbN$ is exactly $\ca(\pi_u(A), \{S_{u,i}\})$ of Proposition \ref{P:multi sta}.

First we show that the canonical $*$-epimorphism $q\colon A\times_\al^n \bbN \rightarrow A_\infty \rtimes_{\al_\infty} \sO_n$ is not injective (cf. Remark \ref{R:cex 1}). To this end, fix a family $(\sigma\otimes I,\{U\otimes T_i\})$ that integrates to a faithful representation of $A_\infty \rtimes_{\al_\infty} \sO_n$. Then the representation implies a number of new relations; for example
\begin{align*}
(U\otimes T_1)^*(\sigma(x)\otimes I)(U \otimes T_2)= (U^*\sigma(x)U)\otimes (T_1^*T_2)=0, \foral x\in A.
\end{align*}
Thus, if $q$ were a $*$-isomorphism, then for all families $(\pi,\{\hat{T_i}\})$ of $A\times_\al^n \bbN$ we would have that $\hat{T_1}^*\pi(x)\hat{T_2}=0$. In particular this should be true for the family $(\id,\{S_i\})$. But
\begin{align*}
S_1^*\id(S_1)S_2=S_1^*S_1S_2=S_2 \neq 0.
\end{align*}

In turn the representation that integrates $(\id, \{S_i\})$ is also not faithful. If it were, then $A \times_\al^n \bbN \simeq \ca(\id, \{S_i\})$, and the  latter is $\sO_n$ since it contains and it is contained in $\sO_n$. Thus, the $*$-epimorphism $q$ above should be faithful, being a representation of the simple $\ca$-algebra $\sO_n$, which leads to a  contradiction.

By Proposition \ref{P:Fmaximal} every representation of $A \times_\al \sT_n^+$ is maximal. Therefore the $\ca$-algebra $A \times_\al^n \bbN$ generated by the universal representation of $A \times_\al \sT_n^+$ is $\cenv(A \times_\al \sT_n^+) \simeq A \times_\al ^n \bbN$.

Thus the canonical embedding $A \times_\al \sT_n^+ \rightarrow A_\infty \rtimes_{\al_\infty} \sO_n$ is completely contractive, but not completely isometric. Otherwise, by Proposition \ref{P:Fmaximal} it would extend to a $*$-isomorphism $A \times_\al^n \bbN \rightarrow A_\infty \rtimes_{\al_\infty} \sO_n$, which is a contradiction. This settles Remark \ref{R:cex 2}.

Finally, referring to Remark \ref{R:cex 3}, if $A \times_\al^t \sT_n^+$ were completely isometrically isomorphic to $A \times_\al \sT_n^+$ via the canonical embedding, then they would have $A_\infty \rtimes_{\al_\infty} \sO_n$ as the same $\ca$-envelope. Therefore $A\times_\al^n \bbN \simeq A_\infty \rtimes_{\al_\infty} \sO_n$, which again is a contradiction.

Another consequence of this counterexample is described in the following Remark.

\begin{remark}
For $(A,\al)$ as above, consider $\ell^1(A,\al,\bbF_n^+)$ as a $\nor{\cdot}_1$-dense subalgebra of $A \times_\al \sT_n^+$. Then the homomorphism $\rho\colon \ell^1(A,\al,\bbF_n^+) \rightarrow A_\infty \rtimes_{\al_\infty} \sO_n$ is completely contractive. Moreover, by Example \ref{E:rpn} it admits a Fourier transform, which implies that $\rho$ is also injective.

However, the homomorphism $\rho\colon \left(\ell^1(A,\al,\bbF_n^+), \nor{\cdot}_1 \right) \rightarrow A_\infty \rtimes_{\al_\infty} \sO_n$ is not completely isometric. If it were, then it would extend to a completely isometry to its $\nor{\cdot}_1$-closure $A \times_\al \sT_n^+$, and we have already argued that this leads to a contradiction.
\end{remark}

\section*{Index}

\noindent $\sT_n^+$: \dotfill{ the tensor algebra of the Toeplitz-Cuntz algebra $\sT_n$}

\noindent $\sO_n^+$: \dotfill{ the non-selfadjoint algebra generated by a Cuntz family $\{S_i\}$}

\noindent $\fA(A,\al)$: \dotfill{ Subsection \ref{Ss:sta-exe}}

\noindent $A \times_\al^1 \bbN$: \dotfill{ Subsection \ref{Ss:sta-exe}}

\noindent $\sT(A,\al,L) \simeq \sT_{\sM_L} \simeq \ca(\rho,S)$: \dotfill{ Subsection \ref{Ss:toeplitz Exel}}

\noindent $A \rtimes_{\al,L} \bbN \simeq \sO(K_\al,\sM_L)$: \dotfill{ Exel's crossed product, Subsection \ref{Ss:crpr exel}}

\noindent $A \rtimes_\al \sO_n$: \dotfill{ Cuntz's twisted crossed product, Subsection \ref{Ss:twisted}}

\noindent $A \times_\al^n \bbN$: \dotfill{ Stacey's crossed product of multiplicity $n$, Subsection \ref{S:multi n sta}}

\noindent $A \times_\al \sT_n^+$: \dotfill{ Subsection \ref{S:multi spr}}

\noindent $A_{\text{nd}} \times_\al \sT_n^+$: \dotfill{ Subsection \ref{S:multi spr}}

\noindent $A \times_\al^t \sT_n^+$: \dotfill{ Subsection \ref{S:multi twi spr}}

\noindent $A_{\textup{nd}} \times_\al^t \sT_n^+$: \dotfill{ Subsection \ref{S:multi twi spr}}

\noindent $\sT_\sX^+$, where $\sX = \sM_L \otimes X_n$: \dotfill{ Subsection \ref{S:tensor}}

\bibliographystyle{plain}

\end{document}